\documentclass[11pt]{amsart}

\usepackage{amsmath,amssymb,amsthm}
\usepackage{thmtools,enumerate}
\usepackage{mathtools}

\usepackage[german,english]{babel}
\usepackage[autostyle]{csquotes}

\usepackage[all]{xy}
\usepackage{pstricks}

\usepackage{tikz-cd}
\usetikzlibrary{babel}

\usepackage{hyperref}
\hypersetup{
	colorlinks=true,
	linkcolor=red,
	citecolor=blue}

\theoremstyle{plain}

\newtheorem{thm}{Theorem}[section]
\newtheorem{prop}[thm]{Proposition}
\newtheorem{lem}[thm]{Lemma}
\newtheorem{cor}[thm]{Corollary}

\newtheorem*{GenNonvanCon}{Generalised Nonvanishing Conjecture}
\newtheorem{thmA}{Theorem}

\newtheorem{corA}[thmA]{Corollary}

\theoremstyle{definition}

\newtheorem{dfn}[thm]{Definition}
\newtheorem{rem}[thm]{Remark}

\newcommand{\Z}{\mathbb{Z}}
\newcommand{\N}{\mathbb{N}}
\newcommand{\Q}{\mathbb{Q}}
\newcommand{\R}{\mathbb{R}}
\newcommand{\C}{\mathbb{C}}
\newcommand{\OO}{\mathcal{O}}

\newcommand{\Rad}[0]{\operatorname{Rad}}
\newcommand{\GL}[0]{\operatorname{GL}}
\newcommand{\Unv}[1]{{#1}^{\rm{univ}}}
\newcommand{\Aut}[0]{\operatorname{Aut}}

\newcommand{\Image}[0]{\operatorname{Im}}

\newcommand{\cal}[1]{\mathcal{#1}}

\DeclareMathOperator{\Pic}{Pic}

\DeclareMathOperator{\rk}{rk}
\DeclareMathOperator{\HHom}{\mathcal{H}\mathit{om}}
\DeclareMathOperator{\End}{\mathcal{E}\mathit{nd}}

\begin{document}
	
	\title[Nonvanishing problem for nef anticanonical bundle]{The Nonvanishing problem for varieties with nef anticanonical bundle}
	
	\author[V.\ Lazi\'c]{Vladimir Lazi\'c}
	\address{Fachrichtung Mathematik, Campus, Geb\"aude E2.4, Universit\"at des Saarlandes, 66123 Saarbr\"ucken, Germany}
	\email{lazic@math.uni-sb.de}
	
	\author[S.\ Matsumura]{Shin-ichi Matsumura}
	\address{Mathematical Institute, Tohoku University, 6-3, Aramaki Aza-Aoba, Aoba-ku, Sendai 980-8578, Japan}
	\email{mshinichi-math@tohoku.ac.jp}

	\author[Th.\ Peternell]{Thomas Peternell}
	\address{Mathematisches Institut, Universit\"at Bayreuth, 95440 Bayreuth, Germany}
	\email{thomas.peternell@uni-bayreuth.de}
	
	\author[N.\ Tsakanikas]{Nikolaos Tsakanikas}
	\address{Institut de Math\'ematiques (CAG), \'Ecole Polytechnique F\'ed\'erale de Lausanne (EPFL), 1015 Lausanne, Switzerland}
	\email{nikolaos.tsakanikas@epfl.ch}
	
	\author[Z.\ Xie]{Zhixin Xie}
	\address{Institut \'Elie Cartan de Lorraine, Universit\'e de Lorraine, F-54506 Nancy, France}
	\email{zhixin.xie@univ-lorraine.fr}

	\thanks{Lazi\'c gratefully acknowledges support by the Deutsche Forschungsgemeinschaft (DFG, German Research
Foundation) -- Project-ID 286237555 -- TRR 195. We thank H.\ Liu, F.\ Meng and the referee for valuable comments. Matsumura is partially supported by Grant-in-Aid for Scientific Research (B) $\sharp$21H00976 and Fostering Joint International Research (A) $\sharp$19KK0342 from JSPS. 
	}
	
	\begin{abstract}
We prove that if $(X,\Delta)$ is a threefold pair with mild singularities such that ${-}(K_X+\Delta)$ is nef, then the numerical class of ${-}(K_X+\Delta)$ is effective.
		\end{abstract}

	\maketitle
	
	\begingroup
		\hypersetup{linkcolor=black}
		\setcounter{tocdepth}{1}
		\tableofcontents
	\endgroup
	
\section{Introduction}
	
One of the pillars of the Minimal Model Program in dimension $3$ over the field of complex numbers is the solution of the Abundance Conjecture in the 1990s. A fundamental part of this result is the Nonvanishing theorem: given a log canonical $ 3 $-fold pair $(X,\Delta)$, if $K_X+\Delta$ is nef, then some multiple of $K_X+\Delta$ has sections. This was first proved by Miyaoka for terminal $3$-folds in \cite{Miy87,Miy88a} and later completed in \cite{KMM94}.
	
Our first main result is that an analogous statement holds at least numerically for $3$-fold pairs $(X,\Delta)$ with ${-}(K_X+\Delta)$ nef.
	
\begin{thmA}\label{thm:mainthm_dim=3}
	Let $(X,\Delta)$ be a projective log canonical pair of dimension $3$ such that ${-}(K_X+\Delta)$ is nef. Assume that $X$ is $\Q$-factorial or that it has rational singularities.\footnote{This holds, for instance, when $(X,\Delta)$ is a klt or, more generally, dlt pair.} 
	Then the following hold:
	\begin{enumerate}[\normalfont (a)]
		\item The numerical class of $ {-}(K_X+\Delta) $ is effective.
	
		\item If $X$ is uniruled, then the numerical class of any nef Cartier divisor on $X$ is effective. 
	\end{enumerate}
\end{thmA}

Together with the Nonvanishing theorem for minimal $3$-folds recalled above, we have:

\begin{corA}\label{cor:cormain}
	Let $(X,\Delta)$ be a projective, $\Q$-factorial, log canonical pair of dimension $3$, let $\varepsilon\in\{{-}1,1\}$, and assume that  $\varepsilon(K_X+\Delta)$ is nef. Then the numerical class of $ \varepsilon(K_X+\Delta) $ is effective.
\end{corA}

When $\varepsilon=1$ in Corollary \ref{cor:cormain}, we actually know that more is true, namely that $\kappa(X,K_X+\Delta)\geq0$. We expect that the same holds for $\varepsilon={-}1$, but this seems to be a very subtle problem. We will return to this question in future work. 

Varieties with mild singularities and nef anticanonical class interpolate between Fano varieties and K-trivial varieties, i.e.\ varieties with numerically trivial canonical class. In those two cases the conclusion of Theorem \ref{thm:mainthm_dim=3}(a) is obvious, whereas the conclusion of Theorem \ref{thm:mainthm_dim=3}(b) is known or is expected to hold. However, not every property of nef divisors on Fano or K-trivial varieties extends to varieties with nef anticanonical class. For instance, the semiampleness of every nef divisor on Fano varieties is known \cite{BCHM} and is expected to hold modulo numerical equivalence in the K-trivial case \cite{LOP16a,LP20a}, while we know that it fails otherwise, for example, if we consider $\mathbb P^2$ blown up at $9$ sufficiently general points. Theorem \ref{thm:mainthm_dim=3} is therefore surprising as it suggests that the birational geometry of varieties with nef \emph{anticanonical} divisor is to a certain extent similar to that of varieties with nef \emph{canonical} divisor.

In many results in this paper we work with pairs $(X,\Delta)$, where $\Delta$ is a $\Q$-divisor. However, most of them hold when $\Delta$ is an $\R$-divisor by Remark \ref{rem:R-divisors}. In particular, Theorem \ref{thm:mainthm_dim=3} holds when $\Delta$ is an $\R$-divisor.

The expectation that Theorem \ref{thm:mainthm_dim=3} might be true formed only very recently, through the study of the following conjecture \cite{LP20a,HanLiu20}.

\begin{GenNonvanCon}
	Let $ (X,\Delta) $ be a projective, $\Q$-factorial, log canonical pair and let $L$ be a nef $\Q$-divisor on $X$ such that $ K_X+\Delta+L $ is pseudoeffective. Then the numerical class of $ K_X+\Delta+L $ is effective.
\end{GenNonvanCon}

This conjecture was introduced for pseudoeffective klt pairs in \cite{LP20a}, where it was also proved in almost all cases for threefolds. In \cite{HanLiu20} the conjecture was formulated in the language of generalised pairs and it was proved in dimension $2$, which implies the analogue of Theorem \ref{thm:mainthm_dim=3} for surfaces, see Theorem \ref{thm:GenNonvanSurfaces}. Those results and the circle of ideas surrounding the Generalised Nonvanishing Conjecture were some of the main inspirations for this paper.

Indeed, if $L$ is a nef Cartier divisor on a projective log canonical pair $(X,\Delta)$ such that ${-}(K_X+\Delta)$ is nef, then one may write $L=K_X+\Delta+M$ for the nef $\Q$-divisor $M:=L-(K_X+\Delta)$. In other words, $L$ is the generalised canonical divisor associated with the generalised pair $(X,\Delta+M)$. Even though this remark is tautological, it provides inspiration and the guiding logic for most of the proofs in this work.\label{page}

In this paper we confirm the Generalised Nonvanishing Conjecture in dimension $ 3 $ in many cases, building on the earlier works \cite{HanLiu20,LP20a}. We stress that Theorem \ref{thm:mainthm_dim=3} is indispensable for the proof of our next result.

\begin{thmA}\label{thm:GNCdim3}
	Let $ (X,\Delta) $ be a projective, $\Q$-factorial, log canonical pair of dimension $3$ and let $L$ be a nef $\Q$-divisor on $X$. Assume that $K_X+\Delta+L$ is pseudoeffective. Then the following hold.
	\begin{enumerate}[\normalfont (a)]
	\item If the numerical class of $L$ is effective, then the numerical class of $ K_X+\Delta+L $ is effective.
	
	\item If $\nu(X,K_X+\Delta+L) \in \{ 0,3 \}$, then the numerical class of $ K_X+\Delta+L $ is effective.\footnote{Here, $\nu$ denotes the numerical dimension, see Section \ref{sec:prelim}.}
	
	\item If $X$ is not uniruled, then the numerical class of $ K_X+\Delta+L $ is effective if $\nu(X,K_X+\Delta)>0$.
	
	\item If $X$ is uniruled, then the numerical class of $ K_X+\Delta+L $ is effective if $\nu(X,K_X+\Delta+L)\neq2$.
	\end{enumerate}		
\end{thmA}

We also obtain several results which go beyond dimension $3$. First, in dimension $4$ we have:

\begin{thmA}\label{thm:mainthm_dim=4}
	Let $X$ be a projective variety with canonical singularities of dimension $4$ such that ${-}K_X$ is nef. If
	\begin{enumerate}[\normalfont (a)]
	\item $\nu(X,{-}K_X)\neq2$, or
	\item $\nu(X,{-}K_X)=2$ and $X$ is not rationally connected, 
	\end{enumerate}
	then the numerical class of $ {-}K_X $ is effective.
\end{thmA}

When the numerical dimension of the relevant divisor is $1$ and the Euler-Poincar\'e characteristic of the structure sheaf is non-zero (which holds, for instance, when the underlying variety is rationally connected), we obtain the following result valid in all dimensions.

\begin{thmA}\label{thm:mainthm_numdim=1partA}
	Let $(X,\Delta)$ be a projective log canonical pair with rational singularities such that ${-}(K_X+\Delta)$ is nef and $\chi(X,\OO_X)\neq0$. If $L$ is a nef Cartier divisor on $X$ of numerical dimension $1$, then the numerical class of $L$ is effective.
\end{thmA}

Applied to rationally connected varieties, Theorem \ref{thm:mainthm_numdim=1partA} is one of the main ingredients in the proofs of Theorems \ref{thm:mainthm_dim=3} and \ref{thm:mainthm_dim=4}.

\medskip

A crucial ingredient in the proofs of the results above is the following reduction in Theorem \ref{thm:reduction} to the case of rationally connected varieties. The key concept is that of locally constant fibrations, see Section \ref{sec:prelim}. By \cite{MW21}, the assumption on the existence of a locally constant fibration in Theorem \ref{thm:reduction} is always satisfied after passing to a quasi-\'etale cover.

\begin{thmA}\label{thm:reduction}
	Let $(X,\Delta)$ be a projective klt pair such that ${-}(K_X+\Delta)$ is nef. Assume that there exists a locally constant MRC fibration $f\colon X\to Y$ with respect to $(X,\Delta)$, where $Y$ is a projective canonical variety with $K_Y\sim0$. Let $F$ be the fibre of $f$. If $\kappa\big(F,{-}(K_F+\Delta|_F)\big)\geq0$, then the numerical class of ${-}(K_X+\Delta)$ is effective.
\end{thmA}

Combining Theorem \ref{thm:reduction} with Theorem \ref{thm:mainthm_numdim=1partA}, we obtain the following result valid in every dimension.

\begin{thmA}\label{thm:mainthm_numdim=1partB}
	Let $(X,\Delta)$ be a projective klt pair such that ${-}(K_X+\Delta)$ is nef of numerical dimension $1$. Then the numerical class of ${-}(K_X+\Delta)$ is effective.
\end{thmA}

\medskip

\noindent{\sc Overview of techniques and related work.}
We make a few brief comments on the techniques involved in the proof of Theorem \ref{thm:mainthm_dim=3}. In higher dimensions the situation becomes much more involved, but the main ideas of the proofs are present already in dimension $3$. A more detailed explanation is given at the beginning of each section.

We may assume that $X$ is uniruled by Lemma \ref{lem:nefanticanimpliesuniruled}. When $X$ is a \emph{smooth} rationally connected threefold with ${-}K_X$ nef, then we know that $\kappa(X,{-}K_X)\geq1$ by \cite{BP04}. A detailed classification of all such $X$ was completed in \cite{Xie20}. Those proofs use smoothness crucially. Here we use completely different methods.

We first treat the case where $X$ is rationally connected and ${-}K_X$ is nef. When $\nu(X,{-}K_X)=2$, then the proof proceeds by analysing the Euler-Poincar\'e characteristic via Hirzebruch-Riemann-Roch and Kawamata-Vieh\-weg vanishing -- this part of the proof is similar to Miyaoka's proof for minimal threefolds of numerical dimension $2$, and the main new ingredient is the recent proof of the pseudoeffectivity of the second Chern class of $X$ in \cite{Ou23}. The proof when $\nu(X,{-}K_X)=1$ is much more involved and the techniques used in this case originated in the recent progress towards the Abundance Conjecture in \cite{LP18a,LP20b}. We recall the recent new proof of Corollary \ref{cor:cormain} when $\varepsilon=1$ in Section \ref{sec:minimal}, since it is relevant also for the proof of Theorem \ref{thm:mainthm_dim=3}. We adapt the Nonvanishing criterion from \cite{LP18a} to the situation in this paper by using another important result of \cite{Ou23} on the subsheaves of sheaves of differential forms, which we slightly generalise in Section \ref{sec:subsheaves}.

We treat next the case where $X$ is uniruled and ${-}K_X$ is nef. To that end, we reduce the problem to the previously proved rationally connected case, by proving Theorem \ref{thm:reduction} in Section \ref{sec:reduction}. The key ingredient is the structure theory of varieties with nef anticanonical bundle which was recently completed in \cite{MW21}, building on \cite{Cao19,CH19,Wan22} and other works.

Finally, to treat the general case, we employ the recently developed Minimal Model Program for generalised pairs as in \cite{HaconLiu21,LT22b}. This is where the idea to use generalised pairs hinted at on page \pageref{page} becomes crucial.

We conclude the introduction with a remark about recent results if one has stronger assumptions than nefness on ${-}K_X$. Namely, if ${-}K_X$ is \emph{strictly nef}, i.e.\ if ${-}K_X$ intersects every curve on $X$ positively, then it is expected that $X$ is a Fano variety. This was confirmed in dimension $3$ in \cite{Ser95,LOWYZ} and in many cases in dimension $4$ in \cite{Liu21}. When $X$ is a smooth, projective, rationally connected fourfold with ${-}K_X$ strictly nef, then $\kappa(X,{-}K_X)\geq0$ by \cite{Liu21}.

\section{Nonvanishing on minimal varieties}\label{sec:minimal}

In this section we briefly sketch the proof of the Nonvanishing theorem for minimal threefolds, i.e.\ we sketch the proof of Corollary \ref{cor:cormain} when $\varepsilon=1$. Therefore, if $(X,\Delta)$ is a log canonical pair such that $K_X+\Delta$ is nef, the goal is to show that the numerical class of $K_X+\Delta$ is effective. As mentioned in the introduction, we know in fact that $\kappa(X,K_X+\Delta)\geq0$. This follows from the proof; however, it is a general fact that for divisors of the form $K_X+\Delta$ in any dimension, the effectivity of their numerical class implies that they have non-negative Iitaka dimension, see \cite[Theorem 1.1]{CKP12}. The content of this section is not new, but the goal is to collect proofs which are spread in different papers.

The first proof of the Nonvanishing theorem for minimal threefold pairs was completed in \cite{KMM94}, building on the fundamental work of Miyaoka for terminal threefolds in \cite{Miy87,Miy88a}. The point of this section is to present a recent new proof, which allows generalisations to higher dimensions.

\medskip

The first useful reduction is to assume that $X$ has rational singularities. This can be achieved by \cite[Theorem 5.22]{KM98} by passing to a dlt blowup as in Theorem \ref{thm:dltblowup}(b). We then proceed in several steps.

\subsection*{Irregular threefolds}

We first assume that $X$ admits a nontrivial morphism to an abelian variety. Then $K_X+\Delta$ is semiample: this is the content of \cite[Corollary 1.2]{Hu16} and \cite[Corollary 3.2]{Laz19}, see also \cite[Lemma 4.1]{LM21}, and the argument is inductive, hence it works in principle also in higher dimensions. One of the main ingredients in the proof of \cite[Corollary 3.2]{Laz19} is the subadditivity of the Iitaka dimension from \cite{KP17}. This part of the proof is similar in spirit to the proof from the 1980s, which relied on the solution of particular cases of Iitaka's conjecture $C_{n,m}$, see for instance \cite[pp.\ 73-74]{MP97}. 

\subsection*{Regular terminal threefolds}

We may thus assume that $H^1(X,\OO_X)=0$. It is convenient to first deal with the case when $\Delta=0$ and the singularities of $ X $ are terminal, and to distinguish four cases, depending on the value of the numerical dimension $\nu(X,K_X)\in\{0,1,2,3\}$. The result is trivial if $\nu(X,K_X)\in\{0,3\}$, since then $K_X$ is either torsion or big.

If $\nu(X,K_X)=2$, then the main ingredients for the proof are Ka\-wa\-ma\-ta-Viehweg vanishing and the pseudoeffectivity of the second Chern class of the minimal variety $X$ \cite[Theorem 6.6]{Miy87} together with a careful analysis of the Euler-Poincar\'e characteristic by the Hirzebruch-Riemann-Roch formula; this is Miyaoka's ingenious proof from \cite{Miy87}. This part of the proof is difficult, but much easier than the case $\nu(X,K_X)=1$.

The main case is $\nu(X,K_X)=1$, as explained in \cite[Remark 6.8]{LP18a}. Since $H^1(X,\OO_X)=0$, we easily deduce that $\chi(X,\OO_X)>0$. Then \cite[Theorem 6.7]{LP18a} gives that $\kappa(X,K_X)\geq0$, and this result works in any dimension. The proof is partly analytic: it uses multiplier ideals associated with singular hermitian metrics. The heart of the proof is the Nonvanishing Criterion \cite[Theorem 6.3]{LP18a}, which reduces the Nonvanishing problem to that of the existence of many pluricanonical differential forms.

Note that Miyaoka's proof in \cite{Miy88a} uses Donaldson's theorem on stable vector bundles and a careful analysis of the algebraic fundamental group of the smooth locus of $X$; it is crucial for this proof that $X$ is a threefold. 

\subsection*{The general case}
Having established the terminal case, we now deal with the general case of a log canonical pair $(X,\Delta)$. By passing to a log resolution, we may assume that $X$ is smooth and $\Delta$ has simple normal crossings support, but we of course lose the nefness of $K_X+\Delta$.

There are two cases to consider. If the variety $X$ is not uniruled, then $K_X$ is pseudoeffective by \cite[Corollary 0.3]{BDPP}, and it suffices to show that $\kappa(X,K_X)\geq0$. By passing to a minimal model of $X$ we may assume that $K_X$ is nef, and then the conclusion follows from the cases treated above. If the variety is uniruled, then the conclusion follows from \cite[Theorem 1.1]{LM21}. As above, the argument is inductive, hence it works in principle also in higher dimensions.

\section{Preliminaries}\label{sec:prelim}

	In this paper we work over the field $ \C $ of complex numbers. We use the same notation for varieties and their associated complex analytic spaces, as well as for coherent sheaves and their associated coherent analytic sheaves.
	
	\medskip
	
	Let $f \colon X\dashrightarrow Y$ be a map between normal complex spaces. Then $f$ is a \emph{fibration} if it is a projective surjective morphism with connected fibres, and it is a \emph{birational contraction} if it is a birational map whose inverse does not contract any divisors.
	
	If $f \colon X\to Y$ is a fibration between normal varieties, then we say that $f$ is \emph{analytically locally trivial} if all fibres of $f$ are isomorphic, and $Y$ can be covered by analytic open sets $U$ such that we have isomorphisms $f^{-1}(U)\simeq U \times F$, where $F$ is the fibre of $f$. In other words, $f\colon X\to Y$ gives $X$ a structure of an analytic fibre bundle over $Y$.
	
	If $f \colon X\to Y$ is a birational contraction between normal varieties and if $D$ is an $\R$-Cartier $\R$-divisor on $X$ such that $f_*D$ is $\R$-Cartier, then we say that $f$ is \emph{$D$-nonpositive} if there exists a resolution of indeterminacies $(p,q)\colon W\to X\times Y$ of $f$ such that $ W $ is smooth and $p^*D\sim_\R q^* f_*D + E$, where $E$ is an effective $q$-exceptional $\R$-divisor.
	
	A morphism $f \colon X\to Y$ between normal varieties is \emph{quasi-\'etale} if $ \dim X = \dim Y $ and if there exists a closed subset $ Z \subseteq X $ of codimension at least $2$ such that $ f |_{X \setminus Z} \colon X \setminus Z \to Y $ is \'etale. A \emph{quasi-\'etale cover} is a finite, surjective, quasi-\'etale morphism.
	
	If $X$ is a normal projective variety, we denote by $ q(X) := h^1(X,\OO_X) $ the \emph{irregularity} of $ X $, and by
	\[ \widetilde{q}(X) := \sup \{ q(Y) \mid Y \text{ is a quasi-\'etale cover of } X \} \in \N \cup \{\infty\} \]
	the \emph{augmented irregularity} of $ X $.
	By \cite[Lemma 2.19]{GGK19}, if $ X $ has klt singularities and numerically trivial canonical class, then $\widetilde{q}(X)\leq\dim X$.
	
	\subsection{Numerical dimension}
	
	Let $X$ be a normal projective variety and let $D$ be a nef $\R$-Cartier $\R$-divisor on $X$. The \emph{numerical dimension} of $D$ is 
	$$\nu(X,D)=\sup\{k\in\N\mid D^k\not\equiv0\} . $$
	One can also define the numerical dimension of $D$ if it is only pseudoeffective \cite{Nak04}. We refer to \cite[\S2.2]{LP20a} for the definition and some main properties that we frequently use in this paper without explicit mention.
	
	We will need the following simple lemma, which is similar in spirit to \cite[Lemma 2.9]{DL15}.
	
	\begin{lem}\label{lem:numdim}
		Let $X$ be a normal projective variety, and let $A$ and $B$ be pseudoeffective $ \Q $-divisors on $X$. Then $\nu(X,\alpha A+\beta B)$ does not depend on the positive rational numbers $\alpha$ and $\beta$.
	\end{lem}
	
	\begin{proof}
		Fix positive rational numbers $\alpha$, $\alpha'$, $\beta$ and $\beta'$. Then there exist positive rational numbers $t_1$ and $t_2$ such that $t_1\alpha\leq \alpha'\leq t_2\alpha$ and $t_1\beta\leq \beta'\leq t_2\beta$. Therefore, 
		\begin{align*}
		\nu(X,\alpha A+\beta B)&=\nu(X,t_1\alpha A+t_1\beta B)\leq\nu(X,\alpha' A+\beta' B)\\
		&\leq\nu(X,t_2\alpha A+t_2\beta B)=\nu(X,\alpha A+\beta B)
		\end{align*}
		by \cite[Proposition V.2.7(1)]{Nak04}.
	\end{proof}
	
	We will also need the following well-known consequence of the usual Kawamata-Viehweg vanishing theorem, see for instance \cite[Lemma 2.1]{LP17} for a short proof.
	
	\begin{lem}\label{lem:KVvanishing}
		Let $(X,\Delta)$ be a projective klt pair of dimension $n$ and let $D$ be a Cartier divisor on $X$ such that $D\sim_\Q K_X+\Delta+L$, where $L$ is a nef $\Q$-divisor with $\nu(X,L)=k$. Then
		$$H^i\big(X,\OO_X(D)\big) =0\quad\text{for all }i>n-k.$$
	\end{lem}

	\subsection{Num-effectivity}
	
	We first recall here the definition of num-effectivity from \cite{LP20a,LP20b}, which plays a central role in this paper.
	
	\begin{dfn}
		We say that an $ \R $-Cartier $\R$-divisor $D$ on a normal projective variety $X$ is \emph{num-effective} if the numerical equivalence class of $D$ belongs to the effective cone of $X$.
	\end{dfn}
	
	It is easy to see that num-effectivity has good descent properties in the presence of $\Q$-factoriality, whereas in the absence of $\Q$-factoriality we often need at least the presence of rational singularities. This is the content of Lemma \ref{lem:descenteff} below, which generalises \cite[Lemmas 2.14 and 2.15]{LP20a} to the setting of $ \R $-divisors. We first need the following.
	
	\begin{lem}\label{lem:key_descent}
		Let $f \colon X \to Y$ be a birational morphism between two normal projective varieties. Assume that $Y$ is $\Q$-factorial or that both $ X $ and $ Y $ have rational singularities. If $ E $ is a numerically trivial $ \R $-Cartier $ \R $-divisor on $ X $, then $ f_* E $ is a numerically trivial $ \R $-Cartier $ \R $-divisor on $ Y $ and we have $ E = f^* f_* E $.
	\end{lem}

	\begin{proof}
		Assume first that $ Y $ is $ \Q $-factorial. Then $ f_* E $ is an $ \R $-Cartier $ \R $-divisor on $ Y $, and by the Negativity Lemma \cite[Lemma 3.39(1)]{KM98} we infer that $ E = f^* f_* E $. It follows now easily by the projection formula that $f_*E$ is numerically trivial.
		
		Assume now that both $ X $ and $ Y $ have rational singularities. Then the pullback map $ f^* \colon \Pic^0(Y) \to \Pic^0(X) $ is an isomorphism. Since $ E\in \Pic^0(X)\otimes\R $ by \cite[Example 1.3.10]{Laz04}, there exists $ G\in \Pic^0(Y)\otimes\R $ such that $ E \sim_\R f^* G $, and thus $ f_* E \sim_\R G $. We conclude as in the previous paragraph.
	\end{proof}
	
	\begin{lem}\label{lem:descenteff}
		Let $f\colon X\to Y$ be a morphism between two projective varieties. Let $D$ be an $\R$-Cartier $\R$-divisor on $Y$ such that there exists an effective $\R$-Cartier $\R$-divisor $G$ on $X$ with $f^*D\equiv G$.
		\begin{enumerate}[\normalfont (a)]
			\item If $f$ is birational and if $Y$ is $\Q$-factorial, then $f_*G$ is an effective $\R$-Cartier $\R$-divisor and we have $D\equiv f_*G$ and $G=f^*f_*G$.
		
			\item If $f$ is birational and if $X$ and $Y$ have rational singularities, then $f_*G$ is an effective $\R$-Cartier $\R$-divisor and we have $D\equiv f_*G$ and $G=f^*f_*G$.
		
			\item If $f$ is finite and surjective, if $X$ and $Y$ are normal and if $Y$ has rational singularities, then $f_*G$ is an effective $\R$-Cartier $\R$-divisor and we have $D\equiv \frac{1}{\deg f}f_*G$.
		\end{enumerate}
	\end{lem}

	\begin{proof}
		Parts (a) and (b) follow immediately from Lemma \ref{lem:key_descent}. For (c), we repeat verbatim the proof of \cite[Lemma 2.15]{LP20a}, except that we invoke (b) instead of \cite[Lemma 2.14]{LP20a}.
	\end{proof}
	
	By repeating verbatim the proof of \cite[Lemma 2.3]{LP20b}, except that we apply Lemma \ref{lem:descenteff}(b) instead of \cite[Lemma 2.14]{LP20a}, we obtain the following useful corollary.
	
	\begin{lem}\label{lem:MMPnumeff}
		Let $\varphi\colon X\dashrightarrow Y$ be a birational contraction between normal projective varieties, where $X$ has rational singularities. Let $D$ be an $\R$-Cartier $\R$-divisor on $X$ such that $\varphi_*D$ is $\R$-Cartier on $Y$ and such that the map $\varphi$ is $D$-nonpositive. If $\varphi_*D$ is num-effective, then $D$ is num-effective.
	\end{lem}					
	
	\subsection{Pairs and generalised pairs}
	
	A \emph{pair} $(X,\Delta)$ consists of a normal variety $X$ and an effective Weil $\R$-divisor $\Delta$ on $ X $ such that the divisor $K_X+\Delta$ is $\R$-Cartier. In this paper we work almost exclusively with pairs with rational coefficients, i.e.\ $ \Delta $ is a Weil $ \Q $-divisor and $ K_X + \Delta $ is a $ \Q $-Cartier divisor.
	
	The standard reference for the singularities of pairs and the foundations of the Minimal Model Program (MMP) is \cite{KM98}, and we use these freely in this paper. 
	
	We will frequently have to improve the singularities of pairs by passing to suitable models.

\begin{thm}\label{thm:dltblowup}
	Let $(X,\Delta)$ be a pair. 
	\begin{enumerate}[\normalfont (a)]
		\item If $(X,\Delta)$ is klt, then there exists a \emph{$\Q$-factorial terminalisation of $(X,\Delta)$}, i.e.\ a $\Q$-factorial terminal pair $(Y,\Gamma)$ together with a proper birational morphism $f\colon Y \to X$ such that $K_Y+\Gamma\sim_\Q f^*(K_X+\Delta)$.
	
		\item If $(X,\Delta)$ is log canonical, then there exists a \emph{dlt blowup of $(X,\Delta)$}, i.e.\ a $\Q$-factorial dlt pair $(Y,\Gamma)$ together with a proper birational morphism $f\colon Y \to X$ such that $K_Y+\Gamma\sim_\Q f^*(K_X+\Delta)$, where the divisor $\Gamma$ is the sum of $f^{-1}_*\Delta$ and all $ f $-exceptional prime divisors.
	\end{enumerate}
\end{thm}

\begin{proof}
	For part (a), see \cite[Corollary 1.4.3]{BCHM} and the paragraph after that result, while part (b) is \cite[Theorem 3.1]{KK10}.
\end{proof}

	Generalised pairs were introduced in \cite{BH14,BZ16} and they play a crucial role in this paper. Here we only recall their definition and refer to e.g.\ \cite{LT22b} for a thorough discussion of their properties and most recent developments.
	
	\begin{dfn}
		A \emph{generalised pair} or \emph{g-pair} $(X/Z,B+M)$ consists of a normal variety $ X $, equipped with projective morphisms
		\[ X' \overset{f}{\longrightarrow} X \longrightarrow Z , \]
		where $ f $ is birational and $ X' $ is a normal variety, $B$ is an effective $ \R $-divisor on $X$, and $ M := f_* M' $, where $M'$ is an $\R$-Cartier $\R$-divisor on $X'$ which is nef over $Z$, such that $ K_X + B + M $ is $ \R $-Cartier. The divisor $ B $ is the \emph{boundary part} and $M$ is the \emph{nef part} of the g-pair $(X/Z,B+M)$.
	\end{dfn}
	
	The variety $X'$ in the definition may always be chosen as a sufficiently high birational model of $ X $. Often we do not refer explicitly to the whole data of a g-pair and simply write $(X/Z,B+M)$, but remember the whole g-pair structure. We work exclusively with g-pairs which are \emph{NQC}, that is, $M'$ is a non-negative linear combination of $\Q$-Cartier divisors on $X'$ which are nef over $ Z $; the acronym NQC stands for \emph{nef $\Q$-Cartier combinations} \cite{HanLi}. 
	
	In most situations in this paper, $M$ will simply be a $\Q$-Cartier divisor, and very often it will be nef itself, so that we may assume that the map $f$ in the definition is an isomorphism. The variety $Z$ will always be assumed to be a point, and we remove it from notation.
				
	One defines singularities of g-pairs in a similar way as for usual pairs, see for instance \cite[Definition 2.2]{LT22b}. With the same notation as above, an important property we need here is that if $M$ is nef, then the g-pair $(X,B+M)$ has log canonical (respectively klt, dlt) singularities if and only if the pair $(X,B)$ has log canonical (respectively klt, dlt) singularities.
		
	\subsection{The MMP for generalised pairs}
		
	The foundations of the MMP for $\Q$-factorial NQC log canonical generalised pairs were established recently in \cite{HaconLiu21,LT22b}. We recall briefly below the results that we need in this paper.
	
	The Cone and Contraction theorems for NQC g-pairs are now known \cite[Theorem 1.3]{HaconLiu21} and are analogous to the Cone theorem for usual pairs: similarly as in the case of usual pairs, an extremal contraction can be either divisorial, flipping, or define a Mori fibre space structure, and these exist and behave as in the usual MMP, see \cite[Theorem 1.2, Subsection 5.5]{HaconLiu21}. Therefore, we may run an MMP for any $ \Q $-factorial NQC log canonical g-pair, and note that $\Q$-factoriality is preserved in any such MMP by \cite[Corollaries 5.20 and 5.21 and Theorem 6.3]{HaconLiu21}.
	
	Now we elaborate on the MMP with scaling for generalised pairs, since we need it in this work. Let $ \big( X, (B+N) + (M+P) \big) $ be a $ \Q $-factorial NQC log canonical g-pair such that the divisor $ K_X + B + N + M + P $ is nef, where $B+N$ is the boundary part and $M+P$ is the nef part. Then by \cite[Corollary 5.22]{HaconLiu21} either $K_X +B +M$ is nef or, if 
	$$\lambda := \inf\{t \geq 0 \mid K_X + B + tN +M + tP\text{ is nef}\}$$
	is the corresponding \emph{nef threshold}, then there exists a $(K_X +B +M)$-negative extremal ray $R$ such that $(K_X +B +\lambda N +M +\lambda P) \cdot R = 0$. In particular, $K_X +B +\lambda N +M +\lambda P$ is nef and it is trivial with respect to the corresponding extremal contraction. By continuing this process, we obtain a \emph{$(K_X+B+M)$-MMP with scaling of $N+P$}. The main result of \cite{LT22b} is that there exists an MMP with scaling which  terminates in several instances which are of importance for the present paper.
	
\subsection{Around Grauert's theorem}

We use the following consequences of Grauert's theorem \cite[Corollary 12.9]{Har77} very often in this work.

\begin{lem}\label{lem:grauert}
	Let $f\colon X\to Y$ be a projective morphism of varieties, let $\mathcal F$ be a coherent sheaf on $X$ which is flat over $Y$ and assume that $h^0(X_y,\mathcal F|_{X_y})$ is constant as a function in $y\in Y$, where $X_y$ is the fibre of $ f $ over $y$. Let $\varphi\colon f^*f_*\mathcal F\to \mathcal F$ be the natural morphism. Then $f_*\mathcal F$ is a locally free sheaf on $X$, and:
	\begin{enumerate}[\normalfont (a)]
		\item $\varphi$ is non-zero if and only if $H^0(X_y,\mathcal F|_{X_y})\neq0$ for some $y\in Y$,
		
		\item $\varphi$ is surjective if and only if $\mathcal F|_{X_y}$ is globally generated for all $y\in Y$,
		
		\item if $\varphi$ is surjective and if for all $y\in Y$ we have $h^0(X_y,\mathcal F|_{X_y})=\rk\mathcal F$, then $\varphi$ is an isomorphism.
	\end{enumerate}
\end{lem}

\begin{proof}
The sheaf $f_*\mathcal F$ is locally free by Grauert's theorem. The proof of (a) and (b) is standard, see for instance \cite[Proposition 28.1.11]{Vak17} for the case of line bundles. For (c), the assumptions imply that $\varphi$ is a surjective morphism of locally free sheaves of the same rank, hence it is generically an isomorphism. Thus, the kernel of $\varphi$ is a torsion subsheaf of $f^*f_*\mathcal F$, which then must be trivial.
\end{proof}

\begin{lem}\label{lem:End}
Let $f\colon X\to Z$ be an analytically locally trivial fibration between normal projective varieties. If $F$ is the fibre $F$ of $f$, assume that $H^1(F,\OO_F)=0$. Let $\mathcal L$ be a line bundle on $X$. Then $\mathcal L|_F$ does not move infinitesimally and $R^pf_*\mathcal L$ is a vector bundle on $Z$ for every $p$. If $\mathcal L|_F\simeq\OO_F$, then $f_*\mathcal L$ is a line bundle on $Z$ and $\mathcal L\simeq f^*f_*\mathcal L$.
\end{lem}

\begin{proof}
Since $\End(\mathcal L|_F)\simeq\mathcal L|_F\otimes\mathcal L^*|_F\simeq\OO_F$, we have $H^1\big(F,\End(\mathcal L|_F)\big)=0$, which shows the first claim by Grauert's theorem. Since $f$ is flat, the second claim follows from Lemma \ref{lem:grauert}.
\end{proof}

\subsection{Slopes with respect to movable classes}

We use freely the foundations of the theory of reflexive sheaves as in \cite{Har80}. If $\mathcal F$ is a reflexive sheaf on a normal variety $X$ and if $m$ is a positive integer, then we denote by
$$\mathcal F^{\boxtimes m}:=(\mathcal F^{\otimes m})^{**}\quad\text{and}\quad \mathcal \bigwedge^{[m]}\mathcal F:=\Big(\bigwedge\nolimits^m\mathcal F\Big)^{**}$$
the \emph{$m$-th reflexive tensor power} and the \emph{$m$-th reflexive exterior power} of $\mathcal F$, respectively. In particular, they coincide with the usual tensor operations on $\mathcal F$ on a \emph{big open} subset of $X$, i.e.\ on an open subset whose complement in $X$ has codimension at least $2$.

Recall from \cite{BDPP} that the \emph{cone of movable curve classes} on a normal projective variety $X$ is the dual of the cone of pseudoeffective divisors on $X$. The semistability theory of torsion-free coherent sheaves on normal $\Q$-factorial projective varieties with respect to movable curve classes was developed in \cite{GKP16} and many results are analogous to the classical theory of semistability with respect to complete intersection curves.

If $\mathcal F$ is a non-zero torsion-free coherent sheaf on a normal, $\Q$-factorial, projective variety $X$ and if $\alpha$ is a non-zero movable curve class on $X$, we denote by
$$\mu_\alpha(\mathcal F):=\frac{c_1(\mathcal F)\cdot\alpha}{\rk\mathcal F}$$
the \emph{slope of $\mathcal F$ with respect to $\alpha$}. We say that $\mathcal F$ is \emph{$\alpha$-semistable}, respectively \emph{$\alpha$-stable}, if $\mu_\alpha(\mathcal G) \leq \mu_\alpha(\mathcal F)$ for any non-zero coherent subsheaf $\mathcal G\subseteq \mathcal F$, respectively if $\mu_\alpha(\mathcal G) < \mu_\alpha(\mathcal F)$ for any non-zero coherent subsheaf $\mathcal G \subseteq \mathcal F$ with $\rk \mathcal G < \rk \mathcal F$. We define
$$\mu_\alpha^{\max} (\mathcal F) := \sup \left\{ \mu_\alpha(\mathcal G) \mid \text{$ 0 \ne \mathcal G \subseteq \mathcal F$ a coherent subsheaf}\, \right\}$$
and
$$\mu_\alpha^{\min}(\mathcal F):=\inf\left\{\mu_\alpha(\mathcal{Q})\mid \mathcal F\twoheadrightarrow\mathcal{Q}\text{ a non-zero torsion-free quotient}\right\}.$$
Then by \cite[Corollary 2.24]{GKP16} there exists a unique \emph{maximal destabilising subsheaf} $\mathcal F'$ of $\mathcal F$. In particular, $\mu_\alpha^{\max}(\mathcal F)=\mu_\alpha(\mathcal F')$.

By \cite[Corollary 2.26]{GKP16} there exists a unique \emph{Harder--Narasimhan filtration\emph} of $\mathcal F$, i.e.\ a filtration
	$$0=\mathcal F_0\subseteq \mathcal F_1\subseteq\dots\subseteq\mathcal F_r=\mathcal F,$$
	where each quotient $\mathcal F_i/\mathcal F_{i-1}$ is torsion-free, $\alpha$-semistable, and the sequence of slopes $\mu_\alpha(\mathcal F_i/\mathcal F_{i-1})$ is strictly decreasing. Moreover, the proof of the existence of the Harder--Narasimhan filtration shows that for each $i$ the sheaf $\mathcal F_i/\mathcal F_{i-1}$ is the maximal destabilising subsheaf of $\mathcal F/\mathcal F_{i-1}$. The quotient $\mathcal F/\mathcal F_{r-1}$ is called the \emph{minimal destabilising quotient} of $\mathcal F$. We will see in Lemma \ref{lem:dual} below that $\mu_\alpha^{\min}(\mathcal F)=\mu_\alpha(\mathcal F/\mathcal F_{r-1})$.

We will need the following well-known lemma, whose proof we include for the lack of reference.

\begin{lem}\label{lem:dual}
	Let $X$ be a normal, $\Q$-factorial, projective variety and let $\alpha$ be a non-zero movable curve class on $X$. Let $\mathcal F$ be a non-zero torsion-free coherent sheaf on $X$ and let
	$$0=\mathcal F_0\subseteq \mathcal F_1\subseteq\dots\subseteq\mathcal F_r=\mathcal F$$
	be the Harder--Narasimhan filtration of $\mathcal F$. Then the following statements hold.
	\begin{enumerate}[\normalfont (a)]
	\item We have $\mu_\alpha^{\max}(\mathcal F)=-\mu_\alpha^{\min}(\mathcal F^*)$.
	
	\item We have $\mu_\alpha(\mathcal F)\geq\mu_\alpha(\mathcal F/\mathcal F_1)$.
	
	\item We have $\mu_\alpha(\mathcal F)\geq\mu_\alpha(\mathcal F/\mathcal F_{r-1})$.
	
	\item Let $\mathcal G$ be a non-zero torsion-free coherent sheaf on $X$. If $\mu_\alpha(\mathcal F/\mathcal F_{r-1})>\mu_\alpha^{\max}(\mathcal G)$, then $\operatorname{Hom}(\mathcal F,\mathcal G)=0$.

	\item For each $k\in\{1,\dots,r\}$ we have $ \mu_\alpha^{\min}(\mathcal F_k)=\mu_\alpha(\mathcal F_k/\mathcal F_{k-1})$.
\end{enumerate}		
\end{lem}

\begin{proof}
We first show (a). Since the functor $\HHom(\,\cdot\,,\OO_X)$ is contravariant and left-exact, every torsion-free quotient $\mathcal F^*\twoheadrightarrow\mathcal Q$ gives rise to an inclusion $\mathcal Q^*\hookrightarrow\mathcal F^{**}$. Therefore,
$$\mu^{\max}_\alpha(\mathcal F)+\mu_\alpha(\mathcal Q)=\mu^{\max}_\alpha(\mathcal F^{**})-\mu_\alpha(\mathcal Q^*)\geq0\quad\text{for every }\mathcal Q,$$
hence
\begin{equation}\label{eq:first}
\mu^{\max}_\alpha(\mathcal F)+\mu_\alpha^{\min}(\mathcal F^*)\geq0.
\end{equation} 
Conversely, every inclusion $\mathcal G\subseteq\mathcal F$ of torsion-free sheaves gives rise to a map $\varphi\colon\mathcal F^*\to\mathcal G^*$, which is surjective on a big open subset of $X$ on which both $\mathcal F$ and $\mathcal G$ are locally free. Since $\mu_\alpha(\operatorname{Im}\varphi)\leq\mu_\alpha(\mathcal G^*)$ by \cite[Proposition 2.12]{GKP16}, we have
$$\mu^{\min}_\alpha(\mathcal F^*)+\mu_\alpha(\mathcal G)\leq\mu_\alpha(\operatorname{Im}\varphi)-\mu_\alpha(\mathcal G^*)\leq0\quad\text{for every }\mathcal G,$$
hence 
\begin{equation}\label{eq:second}
\mu^{\max}_\alpha(\mathcal F)+\mu_\alpha^{\min}(\mathcal F^*)\leq0.
\end{equation} 
Then (a) follows from \eqref{eq:first} and \eqref{eq:second}.

\medskip

For (b), we have $\mu_\alpha(\mathcal F_1)\geq\mu_\alpha(\mathcal F)$ since $\mathcal F_1$ is the maximal destabilising subsheaf of $\mathcal F$. An easy calculation shows that this is equivalent to $\mu_\alpha(\mathcal F)\geq\mu_\alpha(\mathcal F/\mathcal F_1)$.

\medskip

Next we show (c) by induction on the rank of $\mathcal F$. The filtration
	$$0=\mathcal F_1/\mathcal F_1\subseteq\mathcal F_2/\mathcal F_1\subseteq\dots\subseteq\mathcal F_r/\mathcal F_1=\mathcal F/\mathcal F_1$$
	is the Harder--Narasimhan filtration of $\mathcal F/\mathcal F_1$, hence by induction we have
	$$\mu_\alpha(\mathcal F/\mathcal F_1)\geq\mu_\alpha(\mathcal F/\mathcal F_{r-1}).$$
The claim now follows from the last inequality and from (b).

\medskip

The proof of (d) is the same as the proof of \cite[Lemma 1.3.3]{HuyLeh10}, by replacing the reduced Hilbert polynomial $p$ with the slope $\mu_\alpha$.

\medskip

Finally, we show (e). Fix $k\in\{1,\dots,r\}$. Then $ \mu_\alpha^{\min}(\mathcal F_k)\leq\mu_\alpha(\mathcal F_k/\mathcal F_{k-1})$ by definition. For the converse inequality, consider a non-zero torsion-free quotient $\mathcal F_k\twoheadrightarrow\mathcal Q$, and let
	$$0=\mathcal Q_0\subseteq \mathcal Q_1\subseteq\dots\subseteq\mathcal Q_m=\mathcal Q$$
	be the Harder--Narasimhan filtration of $\mathcal Q$. Then the composite $\mathcal F_k\twoheadrightarrow\mathcal Q\twoheadrightarrow\mathcal Q/\mathcal Q_{m-1}$ is a non-trivial element of $\operatorname{Hom}(\mathcal F_k,\mathcal Q/\mathcal Q_{m-1})$. Note that
	$$0=\mathcal F_0\subseteq \mathcal F_1\subseteq\dots\subseteq\mathcal F_k$$
	is the Harder--Narasimhan filtration of $\mathcal F_k$. Therefore, part (d), the fact that $\mathcal Q/\mathcal Q_{m-1}$ is $\alpha$-semistable and part (c) imply that
	$$ \mu_\alpha(\mathcal F_k/\mathcal F_{k-1})\leq\mu_\alpha^{\max}(\mathcal Q/\mathcal Q_{m-1})=\mu_\alpha(\mathcal Q/\mathcal Q_{m-1})\leq\mu_\alpha(\mathcal Q).$$
	As this holds for any such quotient $\mathcal Q$, we obtain $ \mu_\alpha(\mathcal F_k/\mathcal F_{k-1})\leq\mu_\alpha^{\min}(\mathcal F_k)$, which finishes the proof.
\end{proof}

\subsection{Notions of flatness}

In this paper we use several notions of flatness.

\begin{dfn}
	Let $X$ be a K\"ahler manifold and let $\mathcal E$ be a holomorphic vector bundle of rank $ r $ on $X$. We say that $ \mathcal{E} $ is \emph{flat} if it is defined by a representation $\pi_1(X)\to \operatorname{GL}(r)$; and that $ \mathcal E $ is \emph{hermitian flat} if it is defined by a representation $\pi_1(X)\to U(r)$, or equivalently, if it admits a flat hermitian connection, see \cite[Proposition I.4.21]{Kob87}. Note that all Chern classes of a flat vector bundle vanish by \cite[Proposition II.3.1]{Kob87}.
\end{dfn}

The following version of flatness was introduced in \cite{DPS94}.

\begin{dfn}
	Let $X$ be a normal projective variety and let $\mathcal E$ be a vector bundle on $X$. We say that $\mathcal E$ is \emph{numerically flat} if both $\mathcal E$ and $\mathcal E^*$ are nef.
\end{dfn}

\begin{rem}\label{rem:numflat_equiv}
	Let $X$ be a normal projective variety and let $\mathcal E$ be a vector bundle on $X$. Then $\mathcal E$ is numerically flat if and only if $\mathcal E$ and $\det\mathcal E^*$ are nef. Indeed, one direction is clear. Conversely, if $\mathcal E$ and $\det\mathcal E^*$ are nef, and if $\mathcal E$ is of rank $r$, then $\mathcal E^* = \bigwedge^{r-1}\mathcal E \otimes \det \mathcal E^*$ is also nef. 
\end{rem}

The next lemma gives an important source of examples for numerically flat vector bundles, which is particularly relevant for this paper.

\begin{lem} \label{lem:numflatsym} 
	Let $Y$ be a normal projective variety, let $\mathcal E$ be a vector bundle of rank $r$ on $Y$ and set $X := \mathbb P(\mathcal E)$ with the canonical projection $\pi\colon X\to Y$. Assume that there exists a positive integer $k$ and a line bundle $\mathcal L$ on $Y$ such that $c_1(\mathcal L)={-}\frac{k}{r}c_1(\mathcal E)$ and that the line bundle $\OO_{\mathbb P(\mathcal E)}(k)\otimes\pi^*\mathcal L$ on $X$ is nef. Then the vector bundle
	$ \operatorname{Sym}^k \mathcal E \otimes \mathcal L $
	on $Y$ is numerically flat.
\end{lem}   

\begin{proof}
	Since 
	$$\rk(\operatorname{Sym}^k\mathcal E)= \binom{r+k-1}{k}\quad\text{and}\quad c_1(\operatorname{Sym}^k\mathcal E)= \binom{r+k-1}{k-1}c_1(\mathcal E),$$
	we have
	\begin{equation}\label{eq:c_1a}
		c_1(\operatorname{Sym}^k \mathcal E \otimes \mathcal L)=c_1(\operatorname{Sym}^k\mathcal E)+ \binom{r+k-1}{k}c_1(\mathcal L)=0.
	\end{equation}
	Since the $\Q$-twisted vector bundle $\mathcal E\langle\frac{1}{k}\mathcal L\rangle$ is nef by assumption, the vector bundle $\operatorname{Sym}^k \mathcal E \otimes \mathcal L$ is nef by \cite[Corollary 6.1.19 and Theorem 6.2.12(iii)]{Laz04}, thus numerically flat by \eqref{eq:c_1a} and by Remark \ref{rem:numflat_equiv}.
\end{proof}

We will use the following more precise version of \cite[Theorem 1.18]{DPS94}.

\begin{thm}\label{thm:DPS94} 
	Let $X$ be a projective manifold and let $\mathcal E$ be a holomorphic vector bundle on $X$. Then $\mathcal E$ is numerically flat if and only if it admits a filtration
	$$0=\mathcal E_0\subseteq \mathcal E_1 \subseteq \ldots\subseteq\mathcal E_p=\mathcal E$$
	by vector subbundles such that the quotients $\mathcal E_k/\mathcal E_{k-1}$ for $k=1,\dots,p$ are hermitian flat and stable with respect to any ample polarisation. In particular, they are defined by irreducible unitary representations of $\pi_1(X)$.
	
	If, moreover, $\pi_1(X)$ is abelian, then each quotient $\mathcal E_k/\mathcal E_{k-1}$ is a numerically trivial line bundle.
\end{thm}

\begin{proof}
	The existence of a filtration as above whose quotients are hermitian flat is \cite[Theorem 1.18]{DPS94}. If we fix any ample polarisation $H$, then we may assume that each quotient $\mathcal E_k/\mathcal E_{k-1}$ is $H$-stable by the proof of that result. By hermitian flatness, all the Chern classes of those quotients vanish, hence the statement about representations follows from \cite[Theorem 5.1 and Remark 5.2]{MR84}. For the last statement, if the group $\pi_1(X)$ is abelian, then each such representation is $1$-dimensional by \cite[Corollary 2.7.17]{Kow14}, hence each sheaf $\mathcal E_k/\mathcal E_{k-1}$ has rank one.
\end{proof}

\subsection{Rational chain connectedness and uniruledness}\label{subsec:RCC}

Here we adopt the definition that a variety $X$ is \emph{rationally chain connected} if any two very general points on $X$ can be joined by a chain of rational curves, and it is \emph{rationally connected} if any two very general points on $X$ can be joined by a rational curve. Over $\C$, these are equivalent to more technical definitions of rational (chain) connectedness by \cite[Proposition IV.3.6]{Kol96}. If $X$ is additionally proper, then it is rationally chain connected if and only if \emph{any two points on $X$} can be joined by a chain of rational curves, see \cite[Corollary IV.3.5]{Kol96}.

By \cite[Corollary 1.5]{HM07a}, if $(X,\Delta)$ is a rationally chain connected dlt pair and if $Y$ is any desingularisation of $X$, then both $X$ and $Y$ are rationally connected. Since $X$ has rational singularities by \cite[Theorem 5.22]{KM98}, this implies that
\begin{equation}\label{eq:557}
H^i(X,\OO_X)=0\quad\text{for all }i>0
\end{equation}
by \cite[Corollary IV.3.8]{Kol96} and Hodge theory. In particular,
\begin{equation}\label{eq:557a}
\chi(X,\OO_X)=1.
\end{equation}

If $X$ is a normal proper variety, then it admits a dominant almost holomorphic map $\pi\colon X \dashrightarrow Z$, called  \emph{maximal rationally chain connected} or \emph{MRCC fibration}, to a smooth variety such that the complete fibres of $\pi$ are rationally chain connected and for a very general fibre $F$ of $\pi$, any rational curve in $X$ which intersects $F$ lies in $F$, see \cite[Section IV.5]{Kol96}. We include the following lemma for the lack of reference.

\begin{lem}\label{lem:MRCC}
	Let $(X,\Delta)$ be a dlt pair and let $\pi\colon X \dashrightarrow Z$ be an MRCC fibration of $X$. Then $Z$ is not uniruled.
\end{lem}

\begin{proof}
Let $\sigma\colon X' \to X$ be a desingularisation and let $\pi'\colon X'\dashrightarrow Z$ be the induced rational map. Then there are very general fibres $F'$ of $\pi'$ and $F$ of $\pi$ such that the map $\sigma_{F'}\colon F'\to F$ is a desingularisation. Note that $(F,\Delta|_F)$ is dlt, hence $F'$ is rationally connected as above.

Assume that $Z$ is uniruled. Then $\pi'$ is not an MRCC fibration by \cite[Corollary 1.4]{GHS03}. In particular, we may assume that there exists a rational curve in $X'$ which intersects $F'$ but does not lie in $F'$. Therefore, the image of this curve is a rational curve in $X$ which meets $F$ but is not contained in $F$, a contradiction.
\end{proof}

The lemma implies, in particular, that $\dim Z=\dim X$ if and only if $X$ is not uniruled, and note that $Z$ is a point if and only if $X$ is rationally chain connected. In the context of the lemma, since the complete fibres of $\pi$ are rationally connected, we say that $\pi$ is a \emph{maximal rationally connected fibration} or \emph{MRC fibration} of $X$.

The following very useful generalisation of \cite[Theorem 11]{KL09} implies that a variety with nef anticanonical sheaf is almost always uniruled.

\begin{lem}\label{lem:nefanticanimpliesuniruled}
	Let $ (X,\Delta) $ be a projective pair such that $ {-}(K_X+\Delta) $ is pseudoeffective. Then $X$ is not uniruled if and only if $X$ is canonical, $\Delta=0$ and $K_X\sim_\Q0$.
\end{lem}

\begin{proof}
	Assume that $X$ is canonical with $K_X\sim_\Q0$, and that $X$ is uniruled. If $\pi\colon X'\to X$ is a desingularisation, then $K_{X'}$ is pseudoeffective and $X'$ is uniruled, which is impossible since $X'$ possesses a free rational curve through a general point, see \cite[Section 4.2]{Deb01}. This proves one direction.
	
	Conversely, assume that $X$ is not uniruled. Assume for contradiction that $\Delta\neq0$, and let $C$ be the intersection of $\dim X-1$ general very ample divisors on $X$. Then $C$ is a smooth curve which avoids the singular locus of $X$, and clearly $\deg\OO_C(K_X+\Delta)\leq0$ and $\deg\OO_C(\Delta)>0$. Therefore, $\deg\OO_C(K_X)<0$ and thus $X$ is uniruled by \cite[Corollary 2]{MM86}, a contradiction.
	
	Therefore, we have $\Delta=0$ and, in particular, $K_X$ is $\Q$-Cartier. If $\pi\colon X'\to X$ is a desingularisation, then we may write
	$$K_{X'} \sim_\Q \pi^*K_X+E,$$
	where the divisor $E$ is $\pi$-exceptional. Since $X'$ is not uniruled, the divisor $K_{X'}$ is pseudoeffective by \cite[Corollary 0.3]{BDPP}, hence $K_X\sim_\Q\pi_*K_{X'}$ is pseudoeffective. Since $ {-}K_X $ is pseudoeffective by assumption, this implies $K_X\equiv0$ and $K_{X'}\equiv E$, but then a result of Lazarsfeld \cite[Corollary 13]{KL09} shows that $E\geq0$, and hence $X$ is canonical. Finally, $K_X$ is torsion by \cite[Theorem 8.2]{Kaw85b}.
\end{proof}

\subsection{Varieties with nef anticanonical sheaf}

We work almost always with pairs $(X,\Delta)$, where $\Delta$ is a $\Q$-divisor. However, as mentioned in the introduction, most results of this paper hold for pairs $(X,\Delta)$, where $\Delta$ is an $\R$-divisor. The following remark explains this.

\begin{rem}\label{rem:R-divisors}
	Let $(X,\Delta)$ be a projective, $\Q$-factorial, log canonical pair such that ${-}(K_X+\Delta)$ is nef and $\Delta$ is an $\R$-divisor, and let $L$ be a nef $\Q$-divisor on $X$. If we are interested in the question whether ${-}(K_X+\Delta)$ or $L$ is num-effective, we may pass to a dlt blowup of $(X,\Delta)$ by Theorem \ref{thm:dltblowup}(b) and by Lemma \ref{lem:descenteff}. In particular, we may assume that $X$ is klt. Then by \cite[Corollary 3.6]{HanLiu21} there exists positive real numbers $r_1,\dots,r_k$ and $\Q$-divisors $\Delta_1,\dots,\Delta_k$ such that $r_1+\dots+r_k=1$, all pairs $(X,\Delta_i)$ are log canonical, all divisors ${-}(K_X+\Delta_i)$ are nef and $K_X+\Delta=\sum_{i=1}^kr_i(K_X+\Delta_i)$. In particular, if Theorem \ref{thm:mainthm_dim=3} holds for rational boundaries, then it holds also for real boundaries.
\end{rem}

Next, we recall the important definition of locally constant fibrations \cite[Definition 2.3]{MW21}. The inclusion of (b) in the definition below is justified by the proof of \cite[Proposition 2.5]{MW21}.

\begin{dfn}\label{m-defi}
	Let $f\colon X \to Y$ be a fibration between normal varieties and let $\Delta$ be a Weil $\Q$-divisor on $X$. Then $f$ is a \emph{locally constant fibration with respect to $(X,\Delta)$} if the following hold:
	\begin{enumerate}[\normalfont (a)]
		\item $f$ is an analytically locally trivial morphism with the fibre $F$,
		
		\item the fibre product $\Unv{Y} \times_{Y} X  $ is isomorphic to the product $\Unv{Y} \times F$ over the universal cover $\Unv{Y}$ of $Y$,
		
		\item every component of $\Delta$ is horizontal over $Y$,
		
		\item there exist a $\pi_1(Y)$-invariant Weil $\Q$-divisor $\Delta_F$ on F and a representation 
		$$ \rho \colon \pi_1(Y)\to\Aut(F) $$
		such that $(X,\Delta)$ is isomorphic to the quotient $(\Unv{Y}\times F,\operatorname{pr}_F^*\Delta_F)/\pi_1(Y)$ over $Y$, where $\operatorname{pr}_F\colon \Unv{Y}\times F\to F$ is the second projection and $\pi_1(Y)$ acts diagonally on $\Unv{Y}\times F$ by
		$$ \gamma  \cdot(y,z):=\big(\gamma  \cdot y, \rho(\gamma  )(z)\big) \quad\text{for }\gamma\in\pi_1(Y)\text{ and }(y,z) \in \Unv{Y}\times F.$$
	\end{enumerate}
\end{dfn}

We now come to the following very recent structure theorem from \cite{MW21,EIM23} for varieties with nef anticanonical class, which is fundamental for our work.

\begin{thm}\label{thm:MW}
	Let $ (X,\Delta) $ be a projective klt pair such that $ {-}(K_X+\Delta) $ is nef. Then:
	\begin{enumerate}[\normalfont (a)]
		\item there exists a quasi-\'etale cover $ \mu \colon X' \to X $ such that $ X' $ admits a locally constant MRC fibration $ \pi \colon X' \to Z $ with respect to $(X',\mu^*\Delta)$, and $Z$ has canonical singularities and is a product of an abelian variety, singular Calabi-Yau varieties and singular irreducible holomorphic symplectic varieties as in \cite[Definition 8.16]{GKP16b},
		
		\item if $F$ is the fibre of $\pi$, then
		$$\nu\big(F,{-}(K_F+\Delta|_F)\big)=\nu\big(X,{-}(K_X+\Delta)\big).$$
	\end{enumerate}
\end{thm}

\begin{proof}
	Part (a) is \cite[Corollary 1.2]{MW21}, which improves on \cite[Corollary 1.3]{CH19}. For completeness, we recall its proof: by \cite[Theorem 1.1]{MW21} there exists a normal projective variety $ \widetilde X $ with a quasi-\'etale cover $ \theta\colon\widetilde X \to X $ such that $ \widetilde X $ admits a locally constant fibration $ \widetilde X \to \widetilde Z $ with respect to $(\widetilde{X},\theta^*\Delta)$, where $ \widetilde Z $ is a projective klt variety with $K_{\widetilde Z}\equiv0$. By \cite[Theorem 1.5]{HP19} there exists a finite quasi-\'etale morphism $Z\to \widetilde Z$ such that $Z$ has canonical singularities and is a product of an abelian variety, singular Calabi-Yau varieties and singular irreducible holomorphic symplectic varieties. We set $X':=\widetilde X\times_{\widetilde Z} Z$. 
	
	Part (b) follows from \cite[Theorem 1.5]{EIM23}, which improves on \cite[Theorem 1.2]{CCM21}. 
\end{proof}

\subsection{Shafarevich maps}

We will need the concept of Shafarevich maps \cite[Definition 3.5]{Kol95} in Section \ref{sec:reduction}.

\begin{dfn}\label{dfn:Shafarevich}
Let $X$ be a compact K\"ahler space and let $H$ be a normal subgroup of $\pi_1(X)$. Then there exists a K\"ahler manifold $\operatorname{Sh}^H(X)$ and an almost holomorphic map $\operatorname{sh}^H_X\colon X \dashrightarrow \operatorname{Sh}^H(X)$, called the \emph{$H$-Shafarevich variety} and the \emph{$H$-Shafarevich map} of $X$ respectively, such that there are countably many proper subvarieties $D_i\subseteq X$ with the property that, for a subvariety $Z\subseteq X$ with $Z \not \subseteq \cup_{i} D_{i}$, the normalisation $Z^n$ of $Z$ satisfies that the set $ \operatorname{sh}^H_X(Z)$ is a point if and only if $\Image\big(\pi_1(Z^n) \to\pi_1(X)\big)\cap H$ has finite index in $\Image\big(\pi_1(Z^n) \to\pi_1(X)\big)$. A $H$-Shafarevich variety is defined up to birational equivalence.
\end{dfn}

\subsection{Generalised Nonvanishing on surfaces}

The main result of \cite{HanLiu20} is that the Generalised Nonvanishing Conjecture holds for surfaces. It is crucial for the proofs in this paper. The following result is \cite[Theorem 1.5 and Corollary 1.6]{HanLiu20}.

\begin{thm}\label{thm:GenNonvanSurfaces} ~
	\begin{enumerate}[\normalfont (a)]
		\item Let $(X,B +M)$ be a projective, log canonical, NQC surface g-pair such that $K_X + B +M$ is pseudoeffective. Then $K_X + B +M$ is num-effective.

		\item Let $(X,B)$ be a projective, log canonical, surface pair such that the divisor ${-}(K_X + B)$ is nef. Then $\kappa\big(X,{-}(K_X+B)\big)\geq0$.
	\end{enumerate}
\end{thm}

\section{Subsheaves of tensor powers of the cotangent sheaf}\label{sec:subsheaves}

In this section we first prove the following small generalisation of one of the main results of the paper \cite{Ou23}, following very closely the proof of \cite[Theorem 1.4]{Ou23}.

\begin{thm}\label{thm:Ou}
	Let $(X,\Delta)$ be a projective, $\Q$-factorial, log canonical pair such that ${-}(K_X+\Delta)$ is nef. Let $m$ be a positive integer and let $\mathcal M$ be a saturated subsheaf of $\big(\Omega_X^{[1]}\big)^{\boxtimes m}$. Then ${-}c_1(\mathcal M)$ is pseudoeffective.
\end{thm} 

\begin{proof}
Assume that ${-}c_1(\mathcal M)$ is not pseudoeffective. Then there exists a movable curve class $\alpha$ on $X$ such that $c_1(\mathcal M)\cdot\alpha>0$, so that 
$$\mu_\alpha^{\max}\big(\big(\Omega_X^{[1]}\big)^{\boxtimes m}\big) > 0.$$ 
Since 
$$\mu_\alpha^{\max}\big(\big(\Omega_X^{[1]}\big)^{\boxtimes m}\big) = m\cdot\mu_\alpha^{\max}\big(\Omega_X^{[1]}\big) = {-}m\cdot\mu_\alpha^{\min}(T_X) $$ 
by \cite[Theorem 4.2]{GKP16} and by Lemma \ref{lem:dual}(a), we obtain
\begin{equation}\label{eq:alpha}
\mu_\alpha^{\min}(T_X) < 0.
\end{equation}
In particular, $T_X$ is not $\alpha$-semistable. 

Now we follow closely the proof of \cite[Theorem 1.4]{Ou23}. Let 
$$0=\mathcal E_0\subsetneq \mathcal E_1 \subsetneq \cdots \subsetneq \mathcal E_r=T_X$$
be the Harder-Narasimhan filtration of $T_X$, where $r\geq 2$. Note that
$$c_1(T_X)\cdot\alpha={-}K_X\cdot\alpha={-}(K_X+\Delta)\cdot\alpha+\Delta\cdot\alpha\geq0,$$
hence by \eqref{eq:alpha}, by \cite[Lemma 2.1]{Ou23} and by Lemma \ref{lem:dual}(e) there exists $k\in \{1,\dots,r-1\}$ such that
\begin{equation}\label{eq:quot}
\mu_\alpha(T_X/\mathcal E_k) <0
\end{equation}
and 
\begin{equation}\label{eq:quot1}
\mu_\alpha^{\min}(\mathcal E_k) = \mu_\alpha(\mathcal E_k/\mathcal E_{k-1})>0.
\end{equation}
Then \eqref{eq:quot} implies
\begin{equation}\label{eq:ineq1}
\alpha\cdot K_X+\alpha\cdot  c_1(\mathcal E_k)>0.
\end{equation}
On the other hand, \eqref{eq:quot1} yields
\begin{align*}
	2\mu_\alpha^{\min}(\mathcal E_k) &= 2\mu_\alpha(\mathcal E_k / \mathcal E_{k-1}) > \mu_\alpha(\mathcal E_k / \mathcal E_{k-1}) \\
 	&> \mu_\alpha(\mathcal E_{k+1} / \mathcal E_k) = \mu_\alpha^{\max}(T_X / \mathcal E_k).
\end{align*}
Therefore, $\mathcal E_k$ is an algebraically integrable foliation by \cite[Proposition 2.2]{Ou23}, so there exists a dominant rational map $f \colon X\dashrightarrow Y$ which induces $\mathcal E_k$ such that $0 < \dim Y<\dim X$. If $\Delta_{\textrm{ver}}$ denotes the part of $\Delta$ which is vertical over $Y$, then \cite[Theorem 1.10]{Ou23} implies that the divisor $K_{\mathcal E_k}-K_X-\Delta_{\textrm{ver}}$ is pseudoeffective, hence $K_{\mathcal E_k}-K_X$ is pseudoeffective. In particular, 
$$\alpha\cdot K_X+\alpha\cdot  c_1(\mathcal E_k)\leq0,$$ 
which contradicts \eqref{eq:ineq1} and finishes the proof. 
\end{proof}

An immediate consequence of Theorem \ref{thm:Ou} is the following important result \cite[Corollary 1.5]{Ou23} on the pseudoeffectivity of the second Chern class of varieties with nef \emph{anticanonical} class. It is the anticanonical analogue of the famous result of Miyaoka \cite[Theorem 6.6]{Miy87} on the pseudoeffectivity of the second Chern class of varieties with nef \emph{canonical} class. We provide the proof for the reader's convenience.

\begin{thm}\label{thm:c2psef}
Let $X$ be a projective, $\Q$-factorial, log canonical variety of dimension $n$ such that the codimension of the singular locus of $X$ is at least $3$ and such that ${-}K_X$ is nef. If $\pi\colon X'\to X$ is a desingularisation, then for any nef divisors $H_1,\dots, H_{n-2}$ on $X$ we have
$$c_2(X') \cdot \pi^*H_1\cdot\ldots\cdot \pi^*H_{n-2} \geq 0.$$
\end{thm}

\begin{proof}
By continuity of intersection products we may assume that each divisor $H_i$ is very ample and that it is general in the linear system $|H_i|$. Theorem \ref{thm:Ou} implies that the reflexive sheaf $T_X$ is generically $(H_1,\dots,H_{n-2})$-semipositive in the sense of \cite[\S6]{Miy87}. Since $c_1(X)\equiv{-}K_X$ is nef, by \cite[Theorem 6.1]{Miy87} we obtain
\begin{equation}\label{eq:c2}
c_2(X) \cdot H_1\cdot\ldots\cdot H_{n-2} \geq 0.
\end{equation}

Set $H_i':=\pi^*H_i$ for all $i=1,\dots,n-2$, and set $S:=H_1\cap\ldots\cap H_{n-2}$ and $S':=H_1'\cap\ldots\cap H_{n-2}'$. Since the codimension of the singular locus of $X$ is at least $3$, we may assume that $S$ and $S'$ are smooth surfaces. We first claim that there exist a locally free sheaf $\mathcal F$ on $S$ and $C^\infty$-decompositions 
\begin{equation}\label{eq:Cinfty}
T_X|_S\simeq T_S\oplus\mathcal F\quad\text{and}\quad T_{X'}|_{S'}\simeq T_{S'}\oplus (\pi|_{S'})^*\mathcal F.
\end{equation}
Indeed, there exists a $C^\infty$-decomposition $T_X|_{H_1}\simeq T_{H_1}\oplus N_{H_1/X}$ on the smooth locus of $X$, as well as a $C^\infty$-decomposition $T_{X'}|_{H_1'}\simeq T_{H_1'}\oplus N_{H_1'/X'}$. By induction on the dimension there exist a locally free sheaf $\mathcal G$ on $S$ and $C^\infty$-decompositions 
$$ T_{H_1}|_S\simeq T_S\oplus\mathcal G\quad\text{and}\quad T_{H_1'}|_{S'}\simeq T_{S'}\oplus (\pi|_{S'})^*\mathcal G. $$
Then \eqref{eq:Cinfty} follows by setting $\mathcal F:=\mathcal G\oplus N_{H_1/X}|_S$.

Now, the birational map $\pi|_{S'}\colon S'\to S$ is a sequence of blowups at smooth points by \cite[Corollary III.4.4]{BHPV04}, hence $c_2(S')\geq c_2(S)$ by Noether's formula. Furthermore, we have $c_1(S)\cdot c_1(\mathcal F)=c_1(S')\cdot (\pi|_{S'})^*c_1(\mathcal F)$ by the projection formula, and $c_2(\mathcal F)=c_2\big((\pi|_{S'})^*\mathcal F\big)$. Therefore, by \eqref{eq:Cinfty} and the Whitney product formula we obtain
\begin{align*}
c_2(T_{X'}|_{S'}) &= c_2(S')+c_1(S')\cdot (\pi|_{S'})^*c_1(\mathcal F)+c_2\big((\pi|_{S'})^*\mathcal F\big)\\
&\geq c_2(S)+c_1(S)\cdot c_1(\mathcal F)+c_2(\mathcal F)=c_2(T_X|_S).
\end{align*}
Since $c_2(T_X|_S)\geq0$ by \eqref{eq:c2}, we obtain $c_2(T_{X'}|_{S'})\geq0$, which was to be proved.
\end{proof}

\section{Proof of Theorem \ref{thm:mainthm_numdim=1partA}}\label{sec:Theorem E}

In this section we prove Theorem \ref{thm:mainthm_numdim=1partA}. The proof combines general ideas from \cite{LP18a,LP20b} with Theorem \ref{thm:Ou}. In particular, singular hermitian metrics on line bundles play a key role. The main idea is to connect the num-effectivity of a nef line bundle to the existence of differential forms with values in multiples of the line bundle. This is done in Theorem \ref{thm:criterion} below. The proof follows closely that from \cite[Section 4]{LP18a}; however, extra care has to be taken, since we may not always work on a resolution, as nefness of the anticanonical divisor is not a birational invariant. 

The reason why working with differential forms is advantageous is that we may then apply the Hard Lefschetz theorem from \cite{DPS01} to study higher cohomology groups of multiples of $L$ in more detail, and thus the corresponding Euler-Poincar\'e characteristic.

We start with the following general lemma.

\begin{lem} \label{lemmafund} 
Let $X$ be a projective manifold and let $\mathcal E$ be a locally free sheaf on $X$. Let $\mathcal L$ be a pseudoeffective line bundle on $X$ which is not numerically trivial, and assume that there exist an infinite subset $\mathcal T\subseteq\Z$ and numerically trivial line bundles $\mathcal A_m$ for every $m\in\mathcal T$ such that
$$ H^0(X,\mathcal E \otimes \mathcal L^{\otimes m}\otimes\mathcal A_m) \neq 0.$$
Then there exist a positive integer $r$, a saturated line bundle $\mathcal M$ in $\bigwedge^r\mathcal E$, an infinite subset $\mathcal S\subseteq\N$ and numerically trivial line bundles $\mathcal B_m$ for every $m\in\mathcal S$ such that 
$$H^0(X,\mathcal M\otimes \mathcal L^{\otimes m}\otimes \mathcal B_m) \neq 0.$$
\end{lem} 

\begin{proof}
Since the proof is analogous to that of \cite[Lemma 4.1]{LP18a}, we only indicate the most important changes to the proof.

As in the proof of \cite[Lemma 4.1]{LP18a}, we may assume that $\mathcal T\subseteq\N$. For every $m \in \mathcal T$ we fix a nontrivial section of $H^0(X,\mathcal E \otimes \mathcal L^{\otimes m}\otimes\mathcal A_m)$, which gives an inclusion 
$$\mathcal L^{\otimes -m}\otimes\mathcal A_m^{-1} \to \mathcal E.$$  
Let $\mathcal F \subseteq\mathcal E$ be the  torsion-free coherent sheaf which is the image of the induced map 
$$\textstyle\bigoplus_{m\in \mathcal T} (\mathcal L^{\otimes -m}\otimes\mathcal A_m^{-1}) \to \mathcal E,$$
and let $r$ be the rank of $\mathcal F$. As in the proof of \cite[Lemma 4.1]{LP18a} we deduce that there exist infinitely many $r$-tuples $(m_1,\dots,m_r)$ such that, if we set $\mathcal B_{m_1+\dots+m_r}:=\mathcal A_{m_1}\otimes\dots\otimes\mathcal A_{m_r}$, then there are inclusions
\begin{equation}\label{eq:inclusion2}
\mathcal L^{\otimes {-}(m_1+\dots+m_r)}\otimes\mathcal B_{m_1+\dots+m_r}^{-1} \to \det\mathcal F\subseteq \bigwedge^r\mathcal E.
\end{equation}
Now, if $\mathcal M$ is the saturation of $\det\mathcal F$ in $\bigwedge^r\mathcal E$, then by \eqref{eq:inclusion2} there exists an infinite set $\mathcal S\subseteq \N$ such that
$$H^0(X,\mathcal M\otimes \mathcal L^{\otimes m}\otimes\mathcal B_m) \neq 0\quad\text{for all }m\in\mathcal S,$$
which proves the lemma.
\end{proof} 

The following Nonvanishing Criterion is analogous to \cite[Theorem 6.3]{LP18a} and \cite[Theorem 7.1]{LP20b}.

\begin{thm} \label{thm:criterion}
	Let $(X,\Delta)$ be a projective, $\Q$-factorial, log canonical pair such that ${-}(K_X+\Delta)$ is nef and let $L$ be a nef Cartier divisor on $X$ such that $\nu(X,L) = 1$. Let $\pi\colon Y\to X$ be a resolution of $X$ and assume that for some non-negative integer $p$ there exist infinitely many integers $m$ and numerically trivial line bundles $\mathcal A_m$ such that
	$$ H^0\big(Y,(\Omega^1_Y)^{\otimes p} \otimes \OO_Y(m\pi^*L)\otimes\mathcal A_m\big) \neq 0. $$
	Then $L$ is num-effective. 
\end{thm} 

\begin{proof}
We apply Lemma \ref{lemmafund} with $\mathcal E := (\Omega^1_Y)^{\otimes p} $ and $\mathcal L := \pi^*\OO_X(L)$. Then there exist:
\begin{enumerate}[(i)]
	\item a positive integer $r$ and a Cartier divisor $M$ on $Y$ such that the line bundle $\OO_Y(M)$ is saturated in $\bigwedge^r(\Omega^1_Y)^{\otimes p}$,

	\item an infinite set $\mathcal S\subseteq\N$,

	\item integral divisors $N_m\geq0$ on $Y$ for $m\in\mathcal S$, and 

	\item numerically trivial Cartier divisors $B_m$ on $Y$ for $m\in\mathcal S$,
\end{enumerate}
such that 
\begin{equation}\label{eq:infmany}
N_m \sim M+m\pi^* L+B_m.
\end{equation}
Set $M':=\pi_*M$. Since $X$ is $\Q$-factorial, there exists a positive integer $k$ such that $kM'$ is Cartier. Then the reflexive sheaves $\big(\pi_*\OO_Y(kM)\big)^{**}$ and $\OO_X(kM')$ coincide on a big open subset of $X$, hence they are the same. Since
\begin{align*}
\OO_X(kM')&=\big(\pi_*\OO_Y(kM)\big)^{**}\subseteq \bigg(\pi_*\Big(\bigwedge\nolimits^r(\Omega^1_Y)^{\otimes p}\Big)^{\otimes k}\bigg)^{**}\\
&=\Big(\bigwedge\nolimits^{[r]}\big(\Omega_X^{[1]}\big)^{\boxtimes p}\Big)^{\boxtimes k}\subseteq\big(\Omega_X^{[1]}\big)^{\boxtimes krp},
\end{align*}
by Theorem \ref{thm:Ou} we infer that the divisor ${-}M'$ is pseudoeffective.

By Lemma \ref{lem:key_descent} for each $m\in\mathcal S$ there exists a numerically trivial $\Q$-Cartier divisor $B_m'$ on $X$ such that $B_m=\pi^*B_m'$. Pushing forward \eqref{eq:infmany} to $X$ we obtain
$$\pi_*N_m\sim M' +mL+B_m',$$ 
and hence
\begin{equation}\label{eq:repeat}
mL\sim\pi_*N_m+({-}M'-B_m')\quad\text{for }m\in\mathcal S.
\end{equation}
Note that the divisor ${-}M'-B_m'$ is pseudoeffective.

Now, if there exists a positive integer $m_0\in\mathcal S$ such that $\pi_*N_{m_0}\neq0$, then we conclude by \cite[Theorem 6.1]{LP18a}. 

Otherwise, there exist positive integers $m_1<m_2$ in $\mathcal S$ such that $\pi_*N_{m_1}=\pi_*N_{m_2}=0$. Then \eqref{eq:repeat} gives
$$(m_2-m_1)L\sim B_{m_1}'-B_{m_2}'\equiv 0,$$
which contradicts the assumption $\nu(X,L)=1$. This finishes the proof. 
\end{proof}

We now deduce a result which is crucial for the proof of Theorem \ref{thm:mainthm_numdim=1partA}.

\begin{cor}\label{cor:cohomologygroups}
	Let $(X,\Delta)$ be a projective, $\Q$-factorial, log canonical pair such that ${-}(K_X+\Delta)$ is nef and let $L$ be a nef Cartier divisor on $X$ such that $\nu(X,L) = 1$. Let $\pi\colon Y\to X$ be a resolution of $X$ and assume that for some non-negative integer $p$ there exist an infinite subset $\mathcal S\subseteq\N$ and numerically trivial line bundles $\mathcal A_m$ for $m\in\mathcal S$ such that
	\begin{equation}\label{eq:788}
		H^p\big(Y,\OO_Y({-}m\pi^*L)\otimes\mathcal A_m\big) \neq 0.
	\end{equation}
	Then $L$ is num-effective.
\end{cor}

\begin{proof}
	Denote $\mathcal L:=\pi^*\OO_X(L)$. If there exist $m\in\mathcal S$ and a singular metric $h_m$ with semipositive curvature current on the line bundle $\mathcal L^{\otimes m}\otimes\mathcal A_m^{-1}$ such that $\mathcal I(h_m)\neq\OO_Y$, then $\mathcal L$ is num-effective by \cite[Theorem 6.5]{LP18a}, and hence $L$ is num-effective by Lemma \ref{lem:descenteff}(a).

	Therefore, if we fix singular metrics $h_m$ with semipositive curvature current on $\mathcal L^{\otimes m}\otimes\mathcal A_m^{-1}$ for each $m\in\mathcal S$, we may assume that
\begin{equation}\label{eq:789}
\mathcal I(h_m)=\OO_Y\quad\text{for all }m\in\mathcal S.
\end{equation}	
Fix a K\"ahler form $\omega$ on $Y$. A version of the Hard Lefschetz theorem \cite[Theorem 0.1]{DPS01}, together with \eqref{eq:789}, implies that for each $m\in\mathcal S$ we have the surjection
	\[
	\xymatrix{ 
	H^0\big(Y,\Omega^p_Y\otimes \mathcal L^{\otimes m}\otimes\mathcal A_m^{-1}\big) \ar[rr]^{\omega^{n-p}\wedge\bullet} && H^{n-p}\big(Y,\Omega^n_Y\otimes \mathcal L^{\otimes m}\otimes\mathcal A_m^{-1}\big).
	}
	\]
	This, together with \eqref{eq:788} and with Serre duality, yields 
	$$H^0\big(Y,\Omega^p_Y\otimes \mathcal L^{\otimes m}\otimes\mathcal A_m^{-1}\big) \neq 0$$
	for all $m\in\mathcal S$, and we conclude by Theorem \ref{thm:criterion}.
\end{proof}

We finally have:

\begin{proof}[Proof of Theorem \ref{thm:mainthm_numdim=1partA}]
	Let $ f \colon (W,\Gamma) \to (X,\Delta) $ be a dlt blowup of $ (X,\Delta) $. Since $ X $ has rational singularities by assumption and $ W $ has rational singularities by \cite[Theorem 5.22]{KM98}, we deduce $ \chi(X,\OO_X) = \chi(W,\OO_W) $. Therefore, we may replace $ (X,\Delta) $ with $ (W,\Gamma) $ and $ L $ with $ f^* L $ by Lemma \ref{lem:descenteff}(b), and hence we may assume that $ (X,\Delta) $ is a $ \Q $-factorial dlt pair.
	
	The rest of the proof is similar to that of \cite[Theorem 6.7]{LP18a}. However, since the result in op.\ cit.\ has different assumptions, we provide below the details for the benefit of the reader.
	
	Arguing by contradiction, we assume that $L$ is not num-effective. Pick a resolution $\pi\colon Y\to X$ of $ X $. Then, by Corollary \ref{cor:cohomologygroups}, for all $p\geq 0$ and for all $m\gg0$ we have
	$$H^p\big(Y,\OO_Y({-}m\pi^*L)\big) = 0,$$
	thus
	$$ \textstyle\chi\big(Y,\OO_Y({-}m\pi^*L)\big) = 0$$
	holds for all $m\gg0$, and consequently for all $m\in\Z$. But then $\chi(Y,\OO_Y) = 0$, and therefore $\chi(X,\OO_X) = 0$, since $X$ has rational singularities. This is a contradiction which finishes the proof.
\end{proof}

\section{Proofs of Theorems \ref{thm:reduction} and \ref{thm:mainthm_numdim=1partB}}\label{sec:reduction}

In this section we first prove Theorem \ref{thm:reduction}. To this end, we will combine the methods from \cite{MW21} with a result on the virtual commutativity  -- in a certain sense -- of the fundamental group $\pi_{1}(Y)$ of a compact klt K\"ahler space $Y$ with numerically trivial canonical divisor. We then deduce Theorem \ref{thm:mainthm_numdim=1partB} as an easy consequence of Theorems \ref{thm:mainthm_numdim=1partA} and \ref{thm:reduction}.

We start with the commutativity result. If $Y$ is a projective klt variety with numerically trivial canonical divisor, then it is expected that $\pi_{1}(Y)$ is a virtually abelian group, i.e.\ it contains an abelian subgroup of finite index. This problem is still open when $\dim Y \geq 4$. We are, however, able to show that any of its linear representations is virtually abelian, which suffices for our purposes. We first need the following lemma.

\begin{lem}\label{lem:dominatinggeneraltype}
Let $Y$ be a compact klt K\"ahler space with numerically trivial canonical divisor. Then there does not exist a finite \'etale cover of $Y$ which dominates a  variety of general type.
\end{lem}

\begin{proof}
Since a finite \'etale cover of $Y$ is again a compact klt K\"ahler space with numerically trivial canonical divisor, it suffices to prove the lemma for $Y$ itself.

To this end, assume that there exists a dominating rational map $ Y\dashrightarrow Z$ to a smooth variety of general type. Take a desingularisation $\pi \colon \overline Y \to Y$ which resolves the indeterminacies, and let $h\colon\overline{Y}\to Z$ be the induced morphism: 
$$\xymatrix{
\overline{Y} \ar[d]_{\pi} \ar[rrd]^{h}&& \\
Y \ar@{.>}[rr]&& Z. 
}
$$
Then the inclusion $h^*\Omega^1_{Z} \subseteq \Omega^1_{\overline{Y}}$ induces the inclusion
\begin{equation}\label{eq:556}
(\pi_{*}h^{*} \Omega^1_{Z})^{**} \subseteq (\pi_{*}\Omega^1_{\overline{Y}})^{**}=\Omega^{[1]}_Y,
\end{equation}
where $\Omega^{[1]}_Y$ is the sheaf of reflexive differentials on $Y$. Fix a K\"ahler class $\alpha$ on $Y$. By \cite[Theorem A]{Gue16} the tangent sheaf $T_Y$ is $\alpha$-semistable, hence so is its dual $\Omega^{[1]}_Y$. If $\mu_\alpha$ denotes the slope with respect to $\alpha$, then \eqref{eq:556} gives
\begin{equation*}
\mu_\alpha \big((\pi_{*}h^{*} \Omega^1_{Z})^{**}\big) \leq \mu_\alpha \big(\Omega^{[1]}_{Y}\big)=0.
\end{equation*}
On the other hand, since $\det \big((\pi_{*}h^{*} \Omega^1_{Z})^{**}\big) =  (\pi_{*}h^{*} \omega_Z)^{**}$ and the line bundle $\omega_Z$ is big, we have $\mu_\alpha \big((\pi_{*}h^{*} \Omega^1_{Z})^{**}\big)>0$, a contradiction which finishes the proof.
\end{proof}

\begin{prop}\label{m-gl_rep}
Let $Y$ be a compact klt K\"ahler space with numerically trivial canonical divisor and let $\tau \colon \pi_{1}(Y) \to \GL(N, \mathbb{C})$ be a linear representation. Then $\Image(\tau)$ is virtually abelian.
\end{prop}

\begin{proof}
The strategy is similar to \cite[Theorem 7.8]{Cam04}, \cite[Proposition 3.9]{LOWYZ} and \cite[Proposition 3.3]{GKP21}, which are based on \cite{Zuo99, Mok92, CCE15}. Here we give a detailed proof. 

\medskip

\emph{Step 1.}
We first note that if we have a finite \'etale cover $\nu\colon Y' \to Y$ together with the induced representation
$$ \tau' \colon \pi_{1}(Y') \xrightarrow{\quad \nu_{*} \quad} \pi_{1}(Y) \xrightarrow{\quad \tau \quad}\GL(N, \mathbb{C}), $$
then $\Image(\tau')$ is virtually abelian if and only if $\Image(\tau)$ is. Therefore, we may and will freely replace $Y$ with a finite \'etale cover throughout the proof. 

\medskip

\emph{Step 2.}
Let $H$ be the Zariski closure of $\Image{(\tau)} \subseteq \GL(N, \mathbb{C})$ and let $\Rad(H)$ be the solvable radical of $H$. Note that $H /\Rad(H)$ is a linear algebraic group by \cite[11.5]{Hum75}, and we fix an embedding $H /\Rad(H) \subseteq \GL(M, \mathbb{C})$. Consider the induced representation
\begin{equation} \label{m-eq1}
\overline{\tau}\colon \pi_1(Y) \stackrel{\tau}{\longrightarrow} H \twoheadrightarrow H/\Rad(H)\hookrightarrow \GL(M, \mathbb{C}). 
\end{equation} 
Then we claim that, after replacing $Y$ with a finite \'etale cover, we may assume that $H$ is connected, that the image $\Image ({\bar \tau}) $ is torsion-free, and that the Zariski closure of $\Image (\overline{\tau})  \subseteq \GL(M, \mathbb{C})$ is semisimple.

Indeed, the linear algebraic group $H$ has only finitely many connected components, and let $H_{0}$ be the connected component of $H$ containing the identity element. Then $H/H_{0}$ is a finite group, and thus we may assume $H$ is connected by passing to the \'etale cover corresponding to the kernel of the composite $ \pi_{1}(Y) \to H \twoheadrightarrow H/H_{0}$. By Selberg's lemma, the image $\Image{(\bar \tau)}$ has a torsion-free subgroup $S$ of finite index. Hence, we may assume that $\Image{(\bar \tau)}$ is torsion-free by passing to the \'etale cover corresponding to the kernel of the composite $ \pi_{1}(Y) \to \Image{(\bar \tau)} \twoheadrightarrow \Image{(\bar \tau)}/S$. The last part of the claim follows since $H/\Rad(H)$ is semisimple.

\medskip

\emph{Step 3.}
Let $K$ be the kernel of the map in \eqref{m-eq1}. 
By \cite[Th\'eor\`eme 1]{CCE15} there exists a smooth $K$-Shafarevich variety $Z$ of general type together with the corresponding $K$-Shafarevich map $\operatorname{sh}^K_Y\colon Y \dashrightarrow Z$. By Lemma \ref{lem:dominatinggeneraltype} we have $\dim Z =0$. Then $Y$ itself is the fibre of $\operatorname{sh}^K_Y$, hence the group $K$ is of finite index in $\pi_1(Y)$ by Definition \ref{dfn:Shafarevich}, and therefore $\Image (\overline{\tau})$ is a finite group. Thus, $\Image (\tau)$ is virtually solvable. By Lemma \ref{lem:dominatinggeneraltype} again and by \cite[Corollaire 4.2]{CCE15}, we conclude that $\Image (\tau)$ is virtually abelian.
\end{proof}

Now, let $f\colon X \to Y$ be a locally constant fibration with the fibre $F$. By the proof of \cite[Proposition 2.5]{MW21} we have the following diagram: 
\begin{equation}\label{loc-fib}
\begin{gathered}\xymatrix@C=40pt@R=30pt{
X  \ar[d]_{f} &  \Unv{Y} \times F \ar[d]_{p} \ar[l]^{\beta \quad }  \ar[r]^{\quad q} & F  \\ 
Y   & \ar[l]^{\alpha} \Unv{Y},  &  \\   
}\end{gathered}
\end{equation}
where the square is Cartesian, and $p$ and $q$ are the natural projections of the product $\Unv{Y} \times F$. 

The following proposition improves on \cite[Lemma 2.4]{MW21} and is the heart of the proof of Theorem \ref{thm:reduction}. Note that the proposition is invalid when $f$ is merely analytically locally trivial.

\begin{prop}\label{m-prop}
Let $f\colon X \to Y$ be a locally constant fibration of normal projective varieties with the fibre $F$ and the representation $\rho\colon\pi_1(Y)\to\Aut(F)$ as in Definition \ref{m-defi}, and assume the notation from the diagram \eqref{loc-fib}. Let $L $ be a Cartier divisor on $X$ such that $f_*\OO_X(L)\neq0$ and assume that $H^1(F,\OO_F) = 0$.
Then there exist a Cartier divisor $D$ on $F$ and 
a $\Q$-Cartier $\Q$-divisor $B$ on $Y$ such that: 
\begin{enumerate}[\normalfont (a)]
\item the line bundle $\OO_{F}(D)$ is $\pi_{1}(Y)$-equivariant and $\beta^* L  \sim_{\mathbb{Q}} \beta^* f^* B + q^* D$,

\item there exists a positive integer $m_0$ such that for any positive integer $m$ divisible by $m_{0}$, the sheaf $f_*\OO_X(m(L-f^* B)) $ is a flat locally free sheaf on $Y$ which is given by the representation
\begin{equation*}
\tau\colon \pi_1(Y)  \to \GL \big(H^0(F, \OO_F(mD)) \big)
\end{equation*}
induced by $\rho$. 
\end{enumerate}
Assume additionally that $Y$ has rational singularities and that $\Image{(\tau)}$ is virtually solvable. 
Then for any positive integer $m$ divisible by $m_{0}$ there exists a flat line bundle $\mathcal M$ on $Y$, depending on $m$, such that 
$$H^0\big(X,\OO_X(m(L-f^* B))\otimes f^*\mathcal M\big)\neq0.$$
\end{prop}

\begin{proof}
In the proof, for a point $\bar y \in \Unv{Y} $ we often identify $\{\bar y\} \times F \subseteq \Unv{Y} \times F$ with $F$ under the isomorphism 
$$q|_{\{\bar y\} \times F}\colon \{\bar y\} \times F \to F.$$ 

We divide the proof in four steps.

\medskip

\emph{Step 1.}
In this step we show (a). The numerical class of the line bundle $\OO_{\{\bar y\} \times F} (\beta^*L) $ on $F$ is independent of the choice of $\bar y \in \Unv{Y} $, hence so is the linear equivalence class of $\OO_{\{\bar y\} \times F} (\beta^{*}L) $ as $H^1(F,\OO_F) = 0$. Then there exists a Cartier divisor $D$ on $F$, independent of $\bar y \in \Unv{Y}$, such that $D \sim  (\beta^*L)|_{\{\bar y\} \times F}$, thus the line bundle
$$ \OO_{\Unv{Y}\times F}(\beta^* L  - q^* D) $$
is trivial on every fibre $\{\bar y\} \times F$ of $p\colon \Unv{Y}\times F \to  \Unv{Y}$. 

By applying the same argument to some ample Cartier divisor $A$ on $X$ instead of $L$, we construct an ample line bundle $D_A$ on $F$. 

Now, let $\gamma \in \pi_1(Y)$  be a closed loop and set $g:=\rho(\gamma) \in \Aut(F)$. Since $X \simeq (\Unv{Y}\times F)/\pi_1(Y)$ over $Y$, the line bundle $ \OO_X(\beta^*L)$ is equivariant under the action of $\pi_1(Y)$ via $\rho\colon\pi_1(Y)\to\Aut(F)$. Hence, we have
$$\beta^* L \sim (\gamma, g)\cdot (\beta^* L),$$ 
where $(\gamma, g) \in \Aut(\Unv{Y}\times F)$. Restricting this relation to $\{\bar y\} \times F$, we obtain 
$$ D \sim (\beta^* L) |_{\{\bar y\} \times F} \sim (\gamma, g)\cdot (\beta^* L)|_{\{\bar y\} \times F} \sim  g\cdot (\beta^* L)|_{\{ \gamma \cdot \bar y \} \times F} \sim g \cdot D. $$
Thus, the line bundle $\OO_{F}(D)$ is $\pi_1(Y)$-equivariant. 

Analogously, the ample line bundle $\OO_{F}(D_A)$ is $\pi_1(Y)$-equivariant, hence the Zariski closure $H$ of $\Image (\rho)$ is a linear algebraic group by \cite[Lemma 2.3]{Bri22}. Then, since $H$ has only finitely many connected components, by the proof of \cite[Proposition 2.4]{KKLV89} there exists a positive integer $m_0$ such that the line bundle $\OO_{F}(m D)$ is $H$-linearisable for any positive integer $m$ divisible by $m_0$.

The quotient $\big(q^* \OO_F(m_0 D) \big)/\pi_1(X)$ of the total space of the line bundle $q^* \OO_F(m_0 D)$ determines a line bundle on $X$. Take a Cartier divisor $L_0$ on $X$ such that 
$$\OO_X(L_0)=\big ( q^* \OO_F(m_0 D) \big)/\pi_1(X).$$ 
Then $\beta^*\OO_X(L_0)\simeq q^*\OO_F(m_0 D)$ by definition, hence
\begin{equation}\label{eq:99}
\beta^*L_0 \sim q^* m_0D.
\end{equation}
Since $ D \sim  (\beta^*L)|_{\{\bar y\} \times F}$, we obtain from \eqref{eq:99} that the line bundle $\OO_X(m_0L-L_0)$ is trivial on any fibre of $f\colon X \to Y$ via the isomorphism
$$\beta|_{\{\bar y\} \times F}\colon \{\bar y\} \times F \to f^{-1}\big(\alpha(\bar y)\big).$$ 
Thus, by Lemma \ref{lem:End}, there exists a Cartier divisor $B'$ on $Y$ such that
\begin{equation}\label{eq:99a}
m_0L-L_0\sim f^*B'.
\end{equation}
Setting $B:=\frac{1}{m_0}B'$, we finish the proof of (a) by \eqref{eq:99} and \eqref{eq:99a}. 

\smallskip

For the remainder of the proof, $m$ is a positive integer divisible by $m_0$.

\medskip

\emph{Step 2.}
In this step we show (b). By (a) we have 
\begin{align*}
p_*\OO_{\Unv{Y}\times F} \big(m\beta^* (L - f^* B)\big)&= p_*\OO_{\Unv{Y}\times F} (mq^* D) \\
&= \OO_{\Unv Y}\otimes_\C H^0\big(F, \OO_F(mD)\big), 
\end{align*}
where the second equality follows since $p$ and $q$ are the projections. On the other hand, by the flat base change we have 
$$ p_*\beta^*\OO_{\Unv{Y}\times F} \big(m(L - f^* B)\big)\simeq \alpha^*f_*\OO_X\big(m(L - f^* B)\big). $$
Thus,
\begin{equation}\label{m-flat}
\alpha^*f_*\OO_X\big(m(L - f^* B)\big) \simeq \OO_{\Unv Y}\otimes_\C H^0\big(F, \OO_F(mD)\big). 
\end{equation}
The fundamental group $\pi_1(Y)$ naturally acts on $\alpha^*f_*\OO_X\big(m(L - f^* B)\big)$, and $f_*\OO_X\big(m(L - f^* B)\big)$ is the quotient of $\alpha^*f_*\OO_X\big(m(L - f^* B)\big)$ by the $\pi_1(Y)$-action from the definition of locally constant fibrations. 
The first part of (b) follows from the equivalent definitions of flat locally free sheaves \cite[Chapter I]{Kob87}.

For the second part of (b), we show that  
the $\pi_{1}(Y)$-action on $$\OO_{\Unv Y}\otimes_\C H^0\big(F, \OO_F(mD)\big)$$
coincides with the linear action induced by $\tau$. 
Indeed, by construction, for a point $y\in Y$ with the corresponding fibre $X_{y}=f^{-1}(y)$ and for any $\bar y \in \alpha^{-1}(y)$, the isomorphism \eqref{m-flat} is obtained by identifying $\OO_{X_y}(mL)$  with  $\OO_F(mD)$  via the isomorphism 
\begin{equation*}
h_{\bar y} \colon X_{y} \xrightarrow[\quad (\beta|_{{\{\bar y\}} \times F})^{-1} \quad ]{\simeq } {\{\bar y\}} \times F \xrightarrow[p|_{{\{\bar y\}} \times F}]{\quad \simeq  \quad} F.
\end{equation*}
The isomorphism $h_{\bar y}$ depends on $\bar y $, but is determined by $\rho$: indeed, for two points $\bar y_1,\bar y_2\in\alpha^{-1}(y)$ consider a closed loop $\gamma \in \pi_1(Y)$ with $\gamma  \cdot \bar y_1=\bar y_2 $. 
Then $h_{\bar y_1} \circ h_{\bar y_2} ^{-1}$ can be described by $\rho (\gamma ) \in \Aut(F)$, hence the claim follows. 

\medskip

\emph{Step 3.}
For the remainder of the proof we show the last claim of the proposition. We first show that we may assume that the Zariski closure $H$ of $\Image{(\tau)}$ is connected and solvable. 

As is Step 2 in the proof of Proposition \ref{m-gl_rep}, the connected component $H_0$ of $H$ containing the identity element is of finite index in $H$. By assumption, there exists a solvable normal subgroup $S \subseteq \Image{(\tau)}$ of finite index. Then the kernel $K$ of the composite $\pi_1(Y) \to \Image{(\tau)} \twoheadrightarrow \Image{(\tau)}/(S\cap H_0)$ is of finite index in $\pi_1(Y)$, hence defines a finite \'etale cover $\alpha \colon Y' \to Y$. Setting $X':=Y' \times_{Y} X$, then $X$ and $X'$ are  the quotients of the fibre product $\Unv{Y} \times_{Y} X  \simeq  \Unv{Y} \times F$ by the $\pi_{1}(Y)$-action and the $K$-action, respectively, hence we obtain the Cartesian diagrams
$$\xymatrix@C=40pt@R=30pt{
X \ar[d]_{f} &  X' \ar[d]_{f'} \ar[l]^{\beta'}  & \Unv{Y} \times F\ar[l]^{\beta''}\ar[d]   
\ar@/_18pt/[ll]_{\beta} \ar[d]_{p} \ar[r]_{\quad q} &F\\ 
Y  & \ar[l]  Y' \ar[l]^{\alpha'} &  \Unv{Y}. \ar[l]^{\alpha''} \ar@/^20pt/[ll]^{\alpha} \ar[l]&
}$$
The induced morphism $f' \colon X' \to Y'$ is a locally constant fibration. The Cartier divisors $\alpha'^* B$ and $D$ satisfy (a) and (b) of the proposition applied to $\alpha'^* L$ and $f'$. In particular, the sheaf $f'_*\beta'^*\OO_{X'}\big(m(L-f^* B)\big) $ is a flat locally free sheaf which is determined by the induced representation
$$ \tau' \colon \pi_1(Y') \xrightarrow{\quad \alpha'_* \quad} \pi_1(Y) \xrightarrow{\quad \tau \quad} \GL \big(H^0(F, \OO_F(mD)) \big). $$
The image $\Image(\tau')$ is connected and solvable by the construction of $Y'$. Therefore, by replacing $f$ with $f'$ and by Lemma \ref{lem:descenteff}(c), we may assume that $\Image(\tau)$ is connected and solvable. Then the Zariski closure of $\Image(\tau)$ is also solvable, see for instance \cite[Exercise 17.2]{Hum75}. 

\medskip

\emph{Step 4.}
Now, by the Lie-Kolchin theorem, see \cite[17.6 or 21.2]{Hum75},
after fixing a suitable basis of $H^0\big(F, \OO_F(mD)\big)$, 
the representation matrices $\tau(\gamma  )\in \GL \big(H^0(F, \OO_F(mD))\big)$ 
are upper triangular for any $\gamma \in \pi_1(Y)$. If $r=\dim H^0\big(F, \OO_F(mD)\big)$, then the representations $\pi_{1}(Y) \to \mathbb{C}^{*}$ obtained from each diagonal component of $\tau(\gamma )$ determine flat line bundles $\mathcal M_1,\dots,\mathcal M_r$ on $Y$. Then, by (b) and by considering sub-matrices of $\tau(\gamma)$, we obtain a filtration of locally free sheaves 
$$ 0=\cal E_{0} \subseteq \cal E_{1} \subseteq \cdots  \subseteq \cal E_{r-1} \subseteq \cal E_{r}=f_*\OO_X\big(m(L-f^* B)\big) $$
such that $\cal E_{k}/\cal E_{k-1}\simeq\cal M_{k}$ for each $k$. In particular, 
$$\OO_Y \subseteq f_*\OO_X\big(m(L-f^* B)\big) \otimes \mathcal M_1^{-1} ,$$
hence $H^0\big(X,\OO_X(m(L-f^* B))\otimes f^*\mathcal M_1^{-1}\big)\neq0$, as desired. 
\end{proof}

We are now ready to prove Theorems \ref{thm:reduction} and \ref{thm:mainthm_numdim=1partB}.

\begin{proof}[Proof of Theorem \ref{thm:reduction}]
In the proof we adopt the notation from the diagram \eqref{loc-fib}. Let $F$ be the fibre of $f$. The pair $(X, \Delta) $ is isomorphic to the quotient  $\big(\Unv{Y}\times F, q^* (\Delta|_F)\big)/\pi_1(Y)$, and $\beta \colon \Unv{Y} \times F \to X$ is an \'etale cover. Hence, $\beta^*\Delta=q^* (\Delta|_F)$ and 
$$ \beta^*K_X\sim_\Q K_{\Unv{Y} \times F}\sim_\Q p^*K_{\Unv{Y}} +  q^*K_F, $$
and therefore
\begin{equation}\label{eq:558}
\beta^*(K_X+\Delta)\sim_\Q p^*K_{\Unv{Y}} +  q^*(K_F+\Delta|_F). 
\end{equation}

Now, we have $q(F)=0$ by \eqref{eq:557}, and by assumption there exists a positive integer $m$ such that the divisor $L:={-}m(K+\Delta)$ is Cartier and $H^0\big(F,\OO_F(L)\big)\neq0$, hence $f_* \OO_X(L)\neq0$. Therefore, there exist Cartier divisors $B$ on $Y$ and $D$ on $F$ as in Proposition \ref{m-prop} applied to $L$ and $f$. By \eqref{eq:558} and by the proof of Proposition \ref{m-prop}(a) one sees immediately that $B\sim{-}mK_Y$ and $D\sim{-}(K_F+\Delta|_F)$. In particular, $B\sim0$, and hence, by Proposition \ref{m-prop}(b), the sheaf $f_*\OO_X(L)$ is a flat locally free sheaf on $Y$ which is given by a representation $\tau\colon \pi_1(Y)  \to \GL \big(H^0(F, \OO_F(D)) \big)$. By Proposition \ref{m-gl_rep} the image $\Image (\tau)$ is virtually abelian, hence $L$ is num-effective by the last claim in Proposition \ref{m-prop}. This finishes the proof. 
\end{proof}

\begin{proof}[Proof of Theorem \ref{thm:mainthm_numdim=1partB}]
	By Theorem \ref{thm:MW}(a) there exists a quasi-\'etale cover $ \mu \colon X' \to X $ such that, if we set $\Delta':=\mu^*\Delta$, then $ X' $ admits a locally constant MRC fibration $ \pi \colon X' \to Z $ with respect to $(X',\Delta')$, where $ Z $ is a projective canonical variety with $ K_Z \sim 0 $. Then
	$$K_{X'}+\Delta'\sim_\Q \mu^*(K_X+\Delta),$$
	and hence $ (X',\Delta') $ is terminal by \cite[Proposition 5.20]{KM98}, ${-}(K_{X'}+\Delta')$ is nef and $ \nu \big( X', {-} (K_{X'}+\Delta')\big) = 1 $. Therefore, we may replace $(X,\Delta)$ with $(X',\Delta')$ by Lemma \ref{lem:descenteff}, and thus we may assume that $ X $ admits a locally constant MRC fibration $ \pi \colon X \to Z $ with respect to $(X,\Delta)$, where $ Z $ is a projective canonical variety with $ K_Z \sim 0 $. Since $ X $ is uniruled by Lemma \ref{lem:nefanticanimpliesuniruled}, we have $ 0 \leq \dim Z < \dim X $.
	
	Let $F$ be the fibre of $\pi$. Then $F$ is rationally connected and 
	$$\nu\big(F,{-}(K_F+\Delta|_F)\big) = \nu\big(X,{-}(K_X+\Delta)\big)=1$$
	by Theorem \ref{thm:MW}(b). Therefore, $\kappa\big(F,{-}(K_F+\Delta|_F)\big)\geq0$ by Theorem \ref{thm:mainthm_numdim=1partA} and by \eqref{eq:557a}, as the numerical and $\Q$-linear equivalence coincide on a rationally connected variety. We conclude by Theorem \ref{thm:reduction}.
\end{proof}

Finally, we prove the following result valid in every dimension, which will be used in the proofs of Theorems \ref{thm:mainthm_dim=3} and \ref{thm:mainthm_dim=4}.

\begin{thm} \label{thm:nu2_pair_uniruled}
	Let $ (X,\Delta) $ be a projective klt pair of dimension $ n $ such that $ {-}(K_X+\Delta) $ is nef. Assume that $ \nu \big(X,{-}(K_X+\Delta) \big) = n-1 $ and that $ X $ is not rationally connected. Then $ {-}(K_X+\Delta) $ is num-effective.
\end{thm}

\begin{proof}
	As in the proof of Theorem \ref{thm:mainthm_numdim=1partB}, we may assume that $ X $ admits a locally constant MRC fibration $ \pi \colon X \to Z $ with respect to $(X,\Delta)$, where $ Z $ is a projective canonical variety with $ K_Z \sim 0 $. Since $ X $ is uniruled by Lemma \ref{lem:nefanticanimpliesuniruled}, but not rationally connected, we have $ 0 < \dim Z < n $. If $ F $ is the fibre of $ \pi $, then by Theorem \ref{thm:MW}(b) we have
$$ \nu \big( F, {-}(K_F + \Delta |_F) \big) = \nu \big( X, {-}(K_X + \Delta)\big) = n-1 . $$
Therefore, the divisor $ {-}(K_F + \Delta |_F) $ is big, so we conclude by Theorem \ref{thm:reduction}.
\end{proof}

\section{Projective bundles}

In this section we prove Theorem \ref{thm:abel} on nonvanishing of nef line bundles on analytic projective bundles $X$, which is of independent interest. The results of this section deal with \emph{any} nef line bundle on $X$ and not exclusively with nef line bundles of the form ${-}(K_X+\Delta)$, and we do not assume anything about the nefness of the anticanonical divisor of $X$. In the proofs we use crucially the theory of numerically flat bundles from \cite{DPS94}.

We start with the following result on the numerical flatness of direct image sheaves.

\begin{prop}  \label{lem:abel0} 
	Let $f\colon X \to Z$ be an analytically locally trivial $\mathbb P^r$-fibration between normal projective varieties. Let $L$ be a nef $f$-ample Cartier divisor on $X$ with $\nu(X,L) = r$. Then $ f_*\OO_X(L)$ is a numerically flat locally free sheaf on $Z$.
\end{prop} 

\begin{proof}
Note that the map $f$ is flat and set $ \mathcal F := f_*\OO_X(L)$. Then $\mathcal F$ is locally free by Lemma \ref{lem:End}. Let $C \subseteq Z$ be an irreducible curve and let $ \nu\colon C' \to C $ be the normalisation. It suffices to show that 
\begin{equation}\label{eq:claim}
\mathcal F':=\nu^*(\mathcal F|_C)\text{ is numerically flat on }C'. 
\end{equation}

To this end, set $Y := f^{-1}(C)$ and $Y':=Y\times_C C'$, and consider the Cartesian diagram
$$\begin{gathered} 
\xymatrix{ 
Y' \ar[r]^\varphi \ar[d]_{\pi} & Y \ar[d]^{f|_Y} \\
C' \ar[r]_\nu & C.
}\end{gathered}$$
Then the morphism $\pi\colon Y' \to C'$ is an analytically locally trivial $\mathbb P^r$-fibration. Since $H^2(C',\OO^*_{C'})=0$, by \cite[p.\ 223]{Ele82} and by GAGA theorems there exists a locally free sheaf $\mathcal E$ of rank $r+1$ on $C'$ such that 
$$Y' = \mathbb P(\mathcal E).$$
Denote by $\xi$ a Cartier divisor on $\mathbb P(\mathcal E)$ corresponding to $\OO_{\mathbb P(\mathcal E)}(1)$, let $F$ be the numerical class of any fibre of $\pi$, and set $L':= \varphi^*(L|_Y)$. By \cite[Exercise III.12.5]{Har77} there exist an integer $k$ and a Cartier divisor $D$ on $C'$ such that 
\begin{equation}\label{eq:X_C}
	L' \sim k\xi+\pi^* D,
\end{equation}
and note that $k>0$ since $L$ is $f$-ample. We have $\xi^r\cdot F=\OO_{\mathbb P^r}(1)^r=1$, and hence $\xi^{r+1}=c_1(\mathcal E)$ by the defining relation of Chern classes of $\mathcal E$. Since $\nu(X,L)=r$, we have $(L')^{r+1}=0$, and hence \eqref{eq:X_C} gives
\begin{equation}\label{eq:rEll}
	\deg D={-}\frac{k}{r+1}c_1(\mathcal E).
\end{equation}
By cohomology and base change \cite[28.1.5]{Vak17} applied to the Cartesian diagram
$$\begin{gathered} 
\xymatrix{ 
Y' \ar[r]^\varphi \ar[d]_{\pi} & Y\ar[r] \ar[d]^{f|_Y} & X \ar[d]^{f} \\
C' \ar[r]_\nu & C\ar[r] & Z 
}
\end{gathered}$$
we have 
$$\pi_*\OO_{Y'}(L')\simeq\mathcal F',$$
which together with \eqref{eq:X_C} gives
$$ \mathcal F' \simeq \operatorname{Sym}^k \mathcal E \otimes \OO_{C'}(D). $$
Then \eqref{eq:claim} follows from this last relation by \eqref{eq:rEll} and by Lemma \ref{lem:numflatsym}(a).
\end{proof}

Now, we prove the main result of this section.

\begin{thm}\label{thm:abel}
Let $X$ be a normal projective variety, let $Z$ be a projective klt variety and let $f\colon X \to Z$ be an analytically locally trivial $\mathbb P^r$-fibration. Let $L$ be a nef $f$-ample Cartier divisor on $X$ with $\nu(X,L) = r$.
\begin{enumerate}[\normalfont (a)]
\item If $\pi_1(Z)$ is virtually abelian, then $L$ is num-effective.
\item If $\pi_1(Z)$ is finite, then $\kappa(X,L)\geq0$.
\item If $K_Z \equiv 0$, then $L$ is num-effective.
\end{enumerate}
\end{thm}

As mentioned in Section \ref{sec:reduction}, it is expected that the fundamental group of a projective klt variety $Z$ with $K_Z \equiv 0$ is virtually abelian, in which case (c) would follow immediately from (a). However, this is currently not known.

\begin{proof}[Proof of Theorem \ref{thm:abel}]
We proceed in several steps.

\medskip

\emph{Step 1.}
In this step we show (a) and (b). Since $\pi_1(Z)$ is virtually abelian, there exists a finite \'etale cover $\zeta\colon \widetilde Z\to Z$ such that $\pi_1(\widetilde Z)$ is abelian. Then $\widetilde Z$ is klt by \cite[Proposition 5.20]{KM98} and if $\sigma\colon \widehat Z \to\widetilde Z$ is a desingularisation of $\widetilde Z$, then $\sigma_*\colon \pi_1(\widehat Z) \to \pi_1(\widetilde Z) $ is an isomorphism by \cite[Theorem 1.1]{Tak03}. Consider the base change diagram
$$\begin{gathered} 
\xymatrix{ 
\widehat X \ar[r]^\xi \ar[d]_{\widehat f} & X \ar[d]^{f} \\
\widehat Z \ar[r]_{\zeta\circ\sigma} & Z.
}
\end{gathered}$$
Then $\widehat f\colon \widehat X\to\widehat Z$ is also an analytically locally trivial $\mathbb P^r$-fibration and the divisor $\widehat L:=\xi^*L$ is a nef $\widehat f$-ample Cartier divisor on $\widehat X$ with $\nu(\widehat X,\widehat L) = r$. Since $X$ is klt as it is locally a product of a smooth variety and a klt variety, we have that $L$ is num-effective if and only if $\widehat L$ is  num-effective by Lemma \ref{lem:descenteff}. By replacing $X$ with $\widehat X$, $Z$ with $\widehat Z$, $f$ with $\widehat f$, and $L$ with $\widehat L$, we may assume that $X$ and $Z$ are smooth, and $\pi_1(Z)$ is abelian; in case (b) we may assume that $Z$ is simply connected.

Now, by Proposition \ref{lem:abel0}, the sheaf $f_*\OO_X(L)$ is a numerically flat locally free sheaf on $Z$. Since the group $\pi_1(Z)$ is abelian, by Theorem \ref{thm:DPS94} there exists a numerically trivial line bundle $\mathcal G \subseteq f_*\OO_X(L)$. Therefore, 
$$H^0\big(X,\OO_X(L) \otimes f^*\mathcal G^{{-}1}\big)\simeq H^0(Z,f_*\OO_X(L)\otimes\mathcal  G^{{-}1})\neq0,$$
which shows (a). Part (b) then follows immediately, since numerical equi\-valence and $\Q$-linear equivalence coincide on $Z$.

\medskip 

For the remainder of the proof we show (c).
	
\medskip

\emph{Step 2.}
We first claim that, after making a quasi-\'etale base change, we may assume that $ Z $ is a canonical variety and $Z \simeq A \times Y$, where
\begin{enumerate}
	\item[(i)$_1$] $A$ is an abelian variety,

	\item[(ii)$_1$] $Y$ is a canonical variety such that $K_Y \sim 0$, $\widetilde{q}(Y)=0$ and any finite dimensional representation of $\pi_1(Y)$ has a finite image. 
\end{enumerate}

Indeed, by \cite[Theorem 1.5]{HP19} there exists a finite quasi-\'etale cover $\zeta\colon \widetilde Z\to Z$ such that $\widetilde Z \simeq A \times Y$, where $A$ is an abelian variety and $Y$ is a canonical variety with $K_Y \sim 0$ and $\widetilde q(Y)=0$. Then by \cite[Theorem 13.6]{GGK19} any finite dimensional representation of $\pi_1(Y)$ has a finite image. The claim follows by replacing the morphism $f$ with a base change by $\zeta$ as in Step 1.

\medskip

\emph{Step 3.}
Let $\theta\colon \widehat Y \to Y$ be a desingularisation of $Y$, set $\widehat Z:=A\times\widehat Y$ and let $\sigma:=\operatorname{id}_A\times\theta$ be the corresponding desingularisation of $Z$. Consider the base change diagram
$$\begin{gathered} 
\xymatrix{ 
\widehat X \ar[r]^\xi \ar[d]_{\widehat f} & X \ar[d]^{f} \\
\widehat Z \ar[r]_{\sigma} & Z.
}
\end{gathered}$$
Then $\widehat f\colon \widehat X\to\widehat Z$ is also an analytically locally trivial $\mathbb P^r$-fibration and the divisor $\widehat L:=\xi^*L$ is a nef $\widehat f$-ample Cartier divisor on $\widehat X$ with $\nu(\widehat X,\widehat L) = r $. We also have that $L$ is num-effective if and only if $\widehat L$ is  num-effective by Lemma \ref{lem:descenteff}. 

Now, by Proposition \ref{lem:abel0} the sheaf 
$$\mathcal F :=\widehat f_*\OO_{\widehat X}(\widehat L)$$
is a numerically flat locally free sheaf on $\widehat Z$. Hence, by Theorem \ref{thm:DPS94} there exists a hermitian flat subbundle of rank $k$, 
$$\mathcal G \subseteq \mathcal F.$$

\medskip

\emph{Step 4.}
Fix a fibre $F\simeq\widehat Y$ of the first projection $\widehat Z\to A$, and denote 
$$\mathcal E:=\mathcal G|_F,$$
viewed as a vector bundle on $\widehat Y$. Since $\mathcal E$ is hermitian flat, it is given by a unitary representation $\rho\colon\pi_1(\widehat Y)\to U(k)$. Since $\theta_*\colon \pi_1(\widehat Y) \to \pi_1(Y) $ is an isomorphism by \cite[Theorem 1.1]{Tak03}, the representation $\rho$ has a finite image by (ii)$_1$. Then the kernel of $\rho$ induces a finite \'etale cover $\eta\colon Y'\to Y$. Consider the base change diagram
$$\begin{gathered} 
\xymatrix{ 
\widehat Y' \ar[r]^{\nu} \ar[d] & \widehat Y \ar[d]^\theta \\
Y' \ar[r]_{\eta} & Y.
}
\end{gathered}$$

In this step we claim:
\begin{enumerate}
	\item[(i)$_2$] the sheaf $\mathcal E$ does not move infinitesimally,

	\item[(ii)$_2$] $\nu^*\mathcal E$ is the trivial vector bundle of rank $k$ on $\widehat Y'$.
\end{enumerate}

Indeed, (ii)$_2$ follows immediately. In particular, the sheaf $\mathcal End(\nu^*\mathcal E)\simeq \nu^*\mathcal End(\mathcal E)$ is also trivial. Since $ 0 \leq q(Y') \leq \widetilde{q} (Y) = 0 $ and since $Y'$ has canonical singularities by \cite[Proposition 5.20]{KM98}, it follows readily that $H^1(\widehat Y',\OO_{\widehat Y'})=0$, and hence
$$H^1\big(\widehat Y',\nu^*\mathcal End(\mathcal E)\big)\simeq H^1(\widehat Y',\OO_{\widehat Y'})^{\oplus k}=0.$$
By the finiteness of $\nu$ and by the projection formula we obtain
$$0=H^1\big(\widehat Y',\nu^*\mathcal End(\mathcal E)\big)\simeq H^1\big(\widehat Y,\nu_*\nu^*\mathcal End(\mathcal E)\big)\simeq H^1\big(\widehat Y,\mathcal End(\mathcal E)\otimes\nu_*\OO_{\widehat{Y}'}\big).$$
As $\nu$ is finite \'etale, the sheaf $\OO_{\widehat Y}$ is a direct summand of $\nu_*\OO_{\widehat Y'}$, hence
$$H^1\big(\widehat Y,\mathcal End(\mathcal E)\big)=0.$$
This shows (i)$_2$.

\medskip

\emph{Step 5.}
Set $\widehat Z':=A\times\widehat Y'$ and $\mu:=\operatorname{id}_A\times\nu\colon \widehat Z'\to\widehat Z$, and let $p_A\colon\widehat Z'\to A$ be the first projection. Setting
$$\mathcal G':=\mu^*\mathcal G,$$
by Step 3 we infer that the restriction of the sheaf $\mathcal G'$ to each fibre of $p_A$ is the trivial vector bundle of rank $k$. Thus, setting $\mathcal H:=(p_A)_*\mathcal G'$, Lemma \ref{lem:grauert} yields
$$\mathcal G'\simeq p_A^*\mathcal H.$$
Then the sheaf $\mathcal H$ is numerically flat, so by Theorem \ref{thm:DPS94} there exists a numerically trivial line bundle $\mathcal M\subseteq \mathcal H$. Therefore, $\mathcal M':=p_A^*\mathcal M$ is a numerically trivial line bundle such that $ \mathcal M' \subseteq \mathcal G'$.

Consider the Cartesian diagram
$$\begin{gathered} 
\xymatrix{ 
\widehat X' \ar[r]^{\psi} \ar[d]_{\widehat f'} & \widehat X \ar[d]^{\widehat f} \\
\widehat Z' \ar[r]_{\mu} & \widehat Z.
}
\end{gathered}$$
Set $\widehat L':=\psi^*\widehat L$ and $\mathcal F':=\mu^*\mathcal F$. By Lemma \ref{lem:descenteff} it suffices to show that $\widehat L'$ is num-effective. The cohomology and base change \cite[28.1.5]{Vak17} yields $\widehat f'_*\OO_{\widehat X'}(\widehat L')=\mathcal F'$, and since $\mathcal M'\subseteq \mathcal F'$,  we obtain
$$H^0\big(\widehat X',\OO_{\widehat X'}(\widehat L') \otimes (\widehat f')^*(\mathcal M')^{{-}1}\big)\simeq H^0\big(\widehat Z',\mathcal F' \otimes(\mathcal  M')^{{-}1}\big)\neq0,$$
which finishes the proof. 
\end{proof}

\section{On generalised pairs}\label{sec:gpairs}

In this section we use the theory of generalised pairs to -- as we will see in Section \ref{sec:TheoremA} -- essentially prove Theorem \ref{thm:mainthm_dim=3}(a) when $X$ has worse than canonical singularities or when $\Delta\neq0$. We will also see that actually the same methods yield the proof of Theorem \ref{thm:mainthm_dim=3}(b).

The rough strategy is the following: by passing to a birational model we may always assume that $\Delta\neq0$. If $L$ is a nef divisor as in Theorem \ref{thm:mainthm_dim=3}(b), then we may write $L=K_X+\Delta+M$ for a nef divisor $M$ on $X$. The method -- also present in a similar form in \cite{HanLiu20,LP20a,LP20b} -- is to modify $L$ by making $\Delta$ and $M$ smaller until the new divisor hits the boundary of the pseudoeffective cone, in which case a slight modification of one of the main results of \cite{LT22b} yields that this new divisor is num-effective; a similar strategy in the context of usual pairs was employed in \cite{LM21}. Some additional technical considerations then allow us to deduce that the divisor $L$ is also num-effective.

We start with several results which posit that, roughly, on a generalised pair $(X,B+M)$ the divisor $K_X+B+M$ is num-effective if it lies on the boundary of the pseudoeffective cone, in a very precise sense.

\begin{thm}\label{thm:GNV_psef_threshold_dlt}
	Assume the existence of minimal models for smooth varieties of dimension $ n-1 $ and assume that the Generalised Nonvanishing Conjecture holds in dimensions at most $ n-1 $.
	
	Let $\big(X,(B+N)+(M+P)\big)$ be a projective NQC $\Q$-factorial dlt g-pair of dimension $n$. Assume that the divisor $ K_X+B+N+M+P $ is pseudoeffective and that for each $\varepsilon>0$ the divisor $K_X+B+M+(1-\varepsilon)(N+P)$ is not pseudoeffective. Then the divisor $ K_X+B+N+M+P $ is num-effective.
\end{thm}

\begin{proof}
	We follow exactly the same strategy as in the proof of \cite[Theorem 3.1]{LT22b}. We indicate here only the necessary modifications and we refer to \cite{LT22b} for details.

	\medskip

	As in Step 1 of the proof of \cite[Theorem 3.1]{LT22b}, but invoking Lemma \ref{lem:MMPnumeff} instead of \cite[Lemma 2.14]{LT22a}, we may assume the following:
	
	\medskip
	
	\emph{Assumption 1}.
	There exists a fibration $\xi \colon X \to Y$ to a normal quasi-projective variety $Y$ with $\dim Y < \dim X$ and such that:
	\begin{enumerate}
		\item[(a$_1$)] $\nu\big(F,(K_X+B+N+M+P)|_F\big)=0$ and $h^1(F,\OO_F)=0$ for a very general fibre $F$ of $\xi$, 
		
		\item[(b$_1$)] $K_X+B+M+(1-\varepsilon) (N+P)$ is not $\xi$-pseudoeffective for any $\varepsilon>0$. 
	\end{enumerate}
	
	\medskip
	
	If $\dim Y=0$, then $ K_X+B+N+M+P $ is num-effective by (a$_1$) and by \cite[Proposition V.2.7(8)]{Nak04}, and we are done.

	\medskip
	
	If $\dim Y>0$, then as in Step 3 of the proof of \cite[Theorem 3.1]{LT22b}, but invoking again Lemma \ref{lem:MMPnumeff} instead of \cite[Lemma 2.14]{LT22a}, we may assume the following:
	
	\medskip
	
	\emph{Assumption 2}.
	There exists a fibration $g \colon X \to T$ to a normal quasi-projective variety $T$ such that:
	\begin{enumerate}
		\item[(a$_2$)] $0<\dim T<\dim X$,
		
		\item[(b$_2$)] $K_X+B+N+M+P\sim_{\R,T} 0$.
	\end{enumerate} 
	However, instead of the g-pair $\big(X,(B+N)+(M+P)\big)$ being $\Q$-factorial dlt, we may now only assume that it is an NQC log canonical g-pair such that $(X,0)$ is $\Q$-factorial klt.
	
	\medskip
	
	By (b$_2$) and by \cite[Theorem 1.2]{HanLiu21} there exists an NQC log canonical g-pair $ (T, B_T + M_T) $ such that
	\begin{equation}\label{eq:GNV_psef_threshold_dlt}
		K_X + B + N + M + P \sim_\R g^* (K_T + B_T + M_T).
	\end{equation}
	Note that the variety $ T $ is $ \Q $-factorial by \cite[Corollary 5.21]{HaconLiu21} and that the divisor $ K_T + B_T + M_T $ is pseudoeffective by \eqref{eq:GNV_psef_threshold_dlt}. Therefore, by (a$_2$) and since we assume the Generalised Nonvanishing Conjecture in dimensions at most $ n-1 $, we infer that the divisor $ K_T + B_T + M_T $ is num-effective. Hence, by \eqref{eq:GNV_psef_threshold_dlt} the divisor $K_X+B+N+M+P$ is num-effective.
\end{proof}

An immediate corollary of Theorem \ref{thm:GNV_psef_threshold_dlt} is the following result, which improves considerably on \cite[Theorem 4.6]{HanLiu20}, especially by removing the assumption on the termination of flips.

\begin{thm}\label{thm:GNV_psef_threshold_lc}
	Assume the existence of minimal models for smooth varieties of dimension $ n-1 $ and assume that the Generalised Nonvanishing Conjecture holds in dimensions at most $ n-1 $.
	
	Let $\big(X,(B+N)+(M+P)\big)$ be a projective NQC $\Q$-factorial log canonical g-pair of dimension $n$. Assume that the divisor $ K_X+B+N+M+P $ is pseudoeffective and that the divisor $K_X+B+M$ is not pseudoeffective. Set
	$$ \tau := \inf \big\{ t \in \R_{\geq 0} \mid K_X+B+M+t(N+P) \text{ is pseudoeffective} \big\} \in (0,1] . $$
	Then the divisor $ K_X+B+M+\tau(N+P) $ is num-effective.
\end{thm}

\begin{proof}
	The proof is almost identical to the proof of \cite[Lemma 3.2]{LT22b}, and we only sketch it here: Take a dlt blowup 
	\[ h \colon \big(X', (B' + \tau h_*^{-1} N) +(M'+\tau P')\big) \to \big(X,(B+\tau N)+(M+\tau P) \big) \]
	of the g-pair $ \big(X,(B+\tau N)+(M+\tau P) \big) $, where $ B' := h_*^{-1} B + E $ and $ E $ is the sum of all $ h $-exceptional prime divisors \cite[Proposition 3.9]{HanLi}. Then we readily see that
	\[ \tau = \inf \{ t \in \R_{\geq 0} \mid K_{X'} + B' + M' + t (h_*^{-1} N+P') \text{ is pseudoeffective} \} . \]
	By Theorem \ref{thm:GNV_psef_threshold_dlt} the divisor $ K_{X'} + B' + \tau h_*^{-1} N + M'+\tau P' $ is num-effective, hence the divisor $ K_X + B + \tau N + M + \tau P $ is num-effective by Lemma \ref{lem:descenteff}(a).
\end{proof}

\begin{rem}
	Theorem \ref{thm:GNV_psef_threshold_lc} also holds if we assume that $ X $ has rational singularities instead of being $ \Q $-factorial by virtue of Lemma \ref{lem:descenteff}(b). Consequently, the same applies to all the following results in this section; in other words, in Theorem \ref{thm:alternative3}, Corollary \ref{cor:non-canonical} and, by extension, Theorem \ref{thm:mainthm_dim=3}, the assumption that $ X $ is $ \Q $-factorial may be replaced by the assumption that $ X $ has rational singularities. 
\end{rem}
	
	Now we can apply these general results in the context of this paper.
	
\begin{thm}\label{thm:alternative3}
	Assume the existence of minimal models for smooth varieties of dimension $n-1$ and assume that the Generalised Nonvanishing Conjecture holds in dimensions at most $n-1$.
	
	Let $(X,\Delta)$ be a projective, $\Q$-factorial, log canonical pair of dimension $n$ such that ${-}(K_X+\Delta)$ is nef and let $L$ be a nef $\Q$-divisor on $X$.
		\begin{enumerate}[\normalfont (i)]
		\item If ${-}K_X$ is num-effective and $K_X\not\equiv0$, then $L$ is num-effective.
		
		\item If $\Delta\neq0$, then $L$ is num-effective.
	\end{enumerate}
\end{thm}
	
\begin{proof}
	We prove simultaneously (i) and (ii). Thus, in what follows, we assume that either that $\Delta\neq0$, or that $\Delta=0$, in which case ${-}K_X$ is num-effective and $K_X\not\equiv0$.

	We may write $L=K_X+\Delta+M$, where the divisor $M:=L-(K_X+\Delta)$ is nef. Then $(X,\Delta+M)$ is a log canonical g-pair, and consider
	$$ \delta := \inf \big\{ t \in \R_{\geq 0} \mid K_X+t\Delta+M \text{ is pseudoeffective} \big\} \in [0,1] . $$
	
	If $\delta>0$, then $K_X+\delta\Delta+M$ is num-effective by Theorem \ref{thm:GNV_psef_threshold_lc}, hence $L=(K_X+\delta\Delta+M)+(1-\delta)\Delta$ is num-effective.
	
	From now on we assume that $\delta=0$. Consider
	$$ \mu := \inf \big\{ t \in \R_{\geq 0} \mid K_X+tM \text{ is pseudoeffective} \big\} \in [0,1] . $$
	Note that $\mu>0$, since otherwise $K_X$ would be pseudoeffective, which would contradict the assumptions in (i) and (ii). Therefore,
	\begin{equation}\label{eq:67}
	K_X+\mu M\quad\text{is num-effective}	
	\end{equation}	 
	by Theorem \ref{thm:GNV_psef_threshold_lc}.
	
	Assume first that $L={-}(K_X+\Delta)$, or equivalently, $M={-}2(K_X+\Delta)$.
	By \eqref{eq:67} there exists an effective $\Q$-divisor $G$ such that $K_X-2\mu (K_X+\Delta)\equiv G$, hence
	$${-}(2\mu-1)(K_X+\Delta)\equiv G+\Delta\geq0.$$
	If $\Delta\neq0$, then $G+\Delta\neq0$, and in particular, $2\mu-1>0$, so $L={-}(K_X+\Delta)$ is num-effective. If $\Delta=0$, then $L={-}K_X$ is num-effective by assumption.
	
	Finally, if $L$ is any nef $\Q$-divisor on $X$, then by the previous paragraph the divisor
	$$\textstyle L=\frac1\mu (K_X+\mu M)-\frac{1-\mu}{\mu}(K_X+\Delta)+\frac1\mu\Delta$$
	is num-effective. This completes the proof.
\end{proof}

\begin{cor}\label{cor:non-canonical}
	Assume the existence of minimal models for smooth varieties of dimension $n-1$ and assume that the Generalised Nonvanishing Conjecture holds in dimensions at most $n-1$.
	
	Let $(X,\Delta)$ be a projective $\Q$-factorial pair of dimension $n$ which is log canonical but not canonical, and assume that ${-}(K_X+\Delta)$ is nef. If $L$ is a nef $\Q$-divisor on $X$, then $L$ is num-effective.
\end{cor}

\begin{proof}
	Assume first that $(X,\Delta)$ is not klt. By Theorem \ref{thm:dltblowup}(b) there exists a dlt blowup $\pi\colon (Y,\Delta_Y)\to(X,\Delta)$ of $(X,\Delta)$. In particular, $(Y,\Delta_Y)$ is dlt, $\Delta_Y\neq0$ and ${-}(K_Y+\Delta_Y)$ is nef. Then $\pi^*L$ is num-effective by Theorem \ref{thm:alternative3}(ii), hence $L$ is num-effective by Lemma \ref{lem:descenteff}(a).
	
	Now assume that $(X,\Delta)$ is klt. By Theorem \ref{thm:dltblowup}(a) there exists a $\Q$-factorial terminalisation $\varphi\colon (Z,\Delta_Z)\to(X,\Delta)$ of $(X,\Delta)$. In particular, $(Z,\Delta_Z)$ is terminal, ${-}(K_Z+\Delta_Z)$ is nef and $\Delta_Z\neq0$ since $(X,\Delta)$ is not canonical. Hence, $\varphi^*L$ is num-effective by Theorem \ref{thm:alternative3}(ii), thus $L$ is num-effective by Lemma \ref{lem:descenteff}(a) again.
\end{proof}

\section{Proof of Theorem \ref{thm:mainthm_dim=3}}\label{sec:TheoremA}

In this section we prove our main result, Theorem \ref{thm:mainthm_dim=3}. As we will see below, the main part of the proof was already done in the previous sections. The following theorem deals with the remaining cases.

\begin{thm} \label{thm:nu2_var_RCchi}  
	Let $X$ be a projective variety with canonical singularities of dimension $3$ such that ${-}K_X$ is nef and $\chi(X,\OO_X)>0$. If $L$ is a nef Cartier divisor on $X$ such that $\nu(X,L) = 2$, then $h^0(X,L)>0$. 
\end{thm} 

As indicated in the introduction, the proof of this result is similar to that of Miyaoka \cite[\S8]{Miy87} for minimal threefolds of numerical dimension $2$. The idea is to combine the Hirzebruch-Riemann-Roch theorem, the pseudoeffectivity of the second Chern class and the general version of the Kawamata-Viehweg vanishing theorem.

\begin{proof}[Proof of Theorem \ref{thm:nu2_var_RCchi}]
By Theorem \ref{thm:dltblowup}(a) and by Lemma \ref{lem:descenteff}, we may assume that $X$ is $\Q$-factorial and terminal, and thus that its singular locus consists of finitely many points by \cite[Corollary 5.18]{KM98}. Let $\pi\colon Y \to X$ be a resolution which is an isomorphism over the smooth locus of $X$. Since $ \nu(Y, \pi^* L) = 2 $, we have $(\pi^*L)^3=0$, while by representing $K_X$ and $L$ as differences of ample divisors which are general in their $\Q$-linear systems, we also see that
\begin{equation}\label{eq:dot3}
	(\pi^*L)^2 \cdot K_Y = (\pi^*L)^2 \cdot \pi^*K_X \quad\text{and}\quad (\pi^*L) \cdot K_Y^2 = (\pi^*L) \cdot (\pi^*K_X)^2.
\end{equation}

Recall that the Todd class of $Y$ is
$$\operatorname{td}(Y)=1-\frac12K_Y+\frac{1}{12}\big(K_Y^2+c_2(Y)\big)+\chi(Y,\OO_Y),$$
and note that $\chi(Y,\OO_Y)=\chi(X,\OO_X)>0$ since $X$ has rational singularities. Then the Hirze\-bruch-Riemann-Roch theorem, together with \eqref{eq:dot3} and since $X$ has rational singularities, gives
\begin{align*}
\chi\big(X,&\OO_X(L)\big)= \chi\big(Y,\OO_Y(\pi^*(L))\big)\\
=& -\frac{1}{4} (\pi^*L)^2\cdot \pi^*K_X+\frac{1}{12} (\pi^*L) \cdot \big((\pi^*K_X)^2+c_2(Y)\big)+\chi(X,\OO_X).
\end{align*}
We have
$$ {-}(\pi^*L)^2\cdot \pi^*K_X\geq0\quad\text{and}\quad (\pi^*L) \cdot (\pi^*K_X)^2\geq0 $$
since $L$ and ${-}K_X$ are nef, and 
$$ (\pi^*L) \cdot c_2(Y)\geq0 $$
by Theorem \ref{thm:c2psef}. Therefore,
\begin{equation}\label{eq:03}
\chi\big(X,\OO_X(L)\big)>0.
\end{equation}
Since 
$$\nu(X,{-}K_X+L)\geq\nu(X,L)=2,$$
we have
\begin{equation}\label{eq:vanish3}
H^i\big(X,\OO_X(L)\big) = H^i\big(X,\OO_X(K_X+({-}K_X+L))\big)= 0\quad\text{ for }i\geq2 
\end{equation}
by Lemma \ref{lem:KVvanishing}. From \eqref{eq:03} and \eqref{eq:vanish3} we obtain
$$ h^0\big(X,\OO_X(L)\big)>0, $$ 
as desired.
\end{proof}

Finally, we have:

\begin{proof}[Proof of Theorem \ref{thm:mainthm_dim=3}]
	If $X$ is not uniruled, then $K_X+\Delta\sim_\Q0$ by Lemma \ref{lem:nefanticanimpliesuniruled}. Therefore, we may assume that $X$ is uniruled, and let $L$ be a nef $\Q$-Cartier divisor on $X$. We will show that $L$ is num-effective.
	
	By Theorem \ref{thm:dltblowup}(b) there exists a dlt blowup $\varphi\colon (Z,\Delta_Z)\to(X,\Delta)$ of $(X,\Delta)$. In particular, $(Z,\Delta_Z)$ is $\Q$-factorial and dlt, and ${-}(K_Z+\Delta_Z)$ is nef. Since it suffices to show that $\varphi^*L$ is num-effective by Lemma \ref{lem:descenteff}, by replacing $(X,\Delta)$ with $(Z,\Delta_Z)$ and $L$ with $\varphi^*L$, we may assume that $X$ is $\Q$-factorial.
	
	By Corollary \ref{cor:non-canonical} and by Theorem \ref{thm:GenNonvanSurfaces}(a) we may assume that the pair $(X,\Delta)$ is canonical. As in the previous paragraph, by Theorem \ref{thm:dltblowup}(a) we may also assume that $(X,\Delta)$ is terminal.
	
	If $\Delta\neq0$, then we conclude by Theorem \ref{thm:alternative3}(ii) and by Theorem \ref{thm:GenNonvanSurfaces}(a).
	
	If $\Delta=0$, then $X$ is a uniruled, $\Q$-factorial, terminal variety such that ${-}K_X$ is nef. In particular, $K_X\not\equiv0$ by Lemma \ref{lem:nefanticanimpliesuniruled}. If $\nu(X,{-}K_X)=3$, then ${-}K_X$ is big, and if $\nu(X,{-}K_X)=1$, then ${-}K_X$ is num-effective by Theorem \ref{thm:mainthm_numdim=1partB}. If $\nu(X,{-}K_X)=2$, then ${-}K_X$ is num-effective by Theorem \ref{thm:nu2_var_RCchi} together with \eqref{eq:557a}, and by Theorem \ref{thm:nu2_pair_uniruled}. Therefore, the divisor $L$ is num-effective by Theorem \ref{thm:alternative3}(i) and by Theorem \ref{thm:GenNonvanSurfaces}(a).
\end{proof}

\section{Proof of Theorem \ref{thm:GNCdim3}}\label{sec:GNV_dim3}

In this section we prove Theorem \ref{thm:GNCdim3}. The proof uses crucially Theorem \ref{thm:mainthm_dim=3} and one of the main results of \cite{LP20a}.

We start with the following proposition which is also of general interest. The result improves \cite[Proposition 4.1]{LP20a} considerably, most notably the assumptions. The proof is also shortened and streamlined by implementing recent progress on generalised pairs \cite{HaconLiu21,LT22b}.

\begin{prop}\label{prop:MMP}
	Assume the existence of minimal models for smooth varieties of dimension $n-1$.

	Let $(X,\Delta)$ be a projective, $\Q$-factorial, log canonical pair of dimension $n$, let $L$ be a nef Cartier divisor on $X$ and let $m>2n$ be a positive integer. Assume that $(X,\Delta)$ has a minimal model or that $K_X + \Delta$ is not pseudoeffective. Then there exists an $L$-trivial $(K_X+\Delta)$-MMP $\varphi\colon X \dashrightarrow Y$ such that either:
	\begin{enumerate}[\normalfont (a)]
		\item $K_Y+\varphi_*\Delta+m\varphi_*L$ is nef, or
		
		\item there exist a Mori fibre space $\theta\colon Y\to Z$ of $K_X+\Delta$ and a Cartier divisor $L_Z$ on $Z$ such that $\varphi_*L\sim\theta^*L_Z$.
			\end{enumerate}
\end{prop}

\begin{proof}
	Note first that if $(X,\Delta)$ has a minimal model, then it has a weak Zariski decomposition. It follows from \cite[Lemma 3.2]{LT22b} that the g-pair $(X,\Delta+mL)$ has a weak Zariski decomposition or that $K_X + \Delta + mL$ is not pseudoeffective, so by \cite[Theorem 1.1]{LT22b} there exists a $(K_X+\Delta+mL)$-MMP with scaling of an ample divisor $ A $ which terminates. It remains to show that this MMP is $L$-trivial; the last part of (b) then follows from the Cone theorem \cite{Amb03,Fuj11a}.
	
	It suffices to show that the MMP is $L$-trivial on the first step. To this end, let 
	$$\lambda:=\min\{t\geq0\mid K_X+\Delta+mL+tA\text{ is nef}\}$$
	be the nef threshold. We are done if $\lambda=0$, hence we may assume that $\lambda>0$. Then there exists a $(K_X+\Delta+mL)$-negative extremal ray $R$ such that 
	$$\big(K_X+\Delta+mL+\lambda A\big)\cdot R=0.$$ 
	In particular, we have $A\cdot R>0$ and
	\begin{equation}\label{eq:negative}
		(K_X+\Delta+mL)\cdot R<0,
	\end{equation}
	and therefore
	$$ (K_X+\Delta)\cdot R<0, $$
	since $L$ is nef. By \cite[Theorem 1]{Kaw91} there exists a curve $C$ whose class belongs to $R$ such that 
	$$(K_X+\Delta)\cdot C\geq-2n.$$ 
	If $L\cdot C>0$, then $L\cdot C\geq 1$ since $L$ is Cartier, hence $(K_X+\Delta+mL)\cdot R>0$, which contradicts \eqref{eq:negative}. Therefore, $L\cdot C=0$, which finishes the proof.
\end{proof}

We now have:

\begin{proof}[Proof of Theorem \ref{thm:GNCdim3}]
	By passing to a log resolution and by Lemma \ref{lem:descenteff}(a) we may assume that $X$ is smooth and $\Delta$ has simple normal crossings support.
	
	\medskip
	
	\emph{Step 1.}
	We first settle (b). If $\nu(X,K_X+\Delta+L)=3$, then $K_X+\Delta+L$ is big, whereas if $\nu(X,K_X+\Delta+L)=0$, then $K_X+\Delta+L$ is num-effective by \cite[Proposition V.2.7(8)]{Nak04}. 
	
	\medskip
	
	\emph{Step 2.}
	Now we treat (c). Since $X$ is not uniruled, the divisor $K_X$ is pseudoeffective by \cite[Corollary 0.3]{BDPP}. Moreover, the pair $\big(X,\Delta-\varepsilon\lfloor\Delta\rfloor\big)$ is klt for any rational number $\varepsilon\in(0,1)$.
	
	We first assume that $\nu\big(X,K_X+\Delta-\varepsilon\lfloor\Delta\rfloor\big)=0$ for all $\varepsilon\in (0,1)\cap\Q$, and we will derive a contradiction. To this end, for each $\varepsilon \in (0,1)\cap\Q$ there exists an effective $\Q$-divisor $D_\varepsilon$ on $ X $ such that $K_X+\Delta-\varepsilon\lfloor\Delta\rfloor\equiv D_\varepsilon$ by \cite[Proposition V.2.7(8)]{Nak04}. In particular, for any fixed $ \varepsilon \in (0,1)\cap\Q $, we deduce that  $D_{\varepsilon/2}\equiv D_\varepsilon+\frac{\varepsilon}{2}\lfloor\Delta\rfloor$, and hence $\nu\big(X,D_\varepsilon+\frac{\varepsilon}{2}\lfloor\Delta\rfloor\big)=0$. Since
	$$\textstyle \nu\big(X,D_\varepsilon+\frac{\varepsilon}{2}\lfloor\Delta\rfloor\big)=\nu\big(X,D_\varepsilon+\varepsilon\lfloor\Delta\rfloor\big)$$
	by Lemma \ref{lem:numdim} and since $K_X+\Delta\equiv D_\varepsilon+\varepsilon\lfloor\Delta\rfloor$, we obtain $\nu(X,K_X+\Delta)=0$, which contradicts our hypothesis.
	
	Thus, there exists a rational number $\varepsilon\in(0,1)$ such that $\nu\big(X,K_X+\Delta-\varepsilon\lfloor\Delta\rfloor\big)>0$, so $K_X+\Delta-\varepsilon\lfloor\Delta\rfloor+L$ is num-effective by \cite[Corollary D]{LP20a}, and thus $K_X+\Delta+L$ is num-effective.
	
	\medskip
	
	\emph{Step 3.}
	If $X$ is not uniruled, then $K_X$ is pseudoeffective by \cite[Corollary 0.3]{BDPP}. Then $\kappa(X,K_X+\Delta)\geq0$ by the Nonvanishing theorem for threefolds, which shows (a) if $X$ is not uniruled. 
	
	\medskip
	
	For the remainder of the proof we assume that $X$ is uniruled, and we settle (d) and complete the proof of (a).
	
	\medskip
	
	\emph{Step 4.}
	Since $X$ is uniruled, it possesses a free rational curve through a general point \cite[Section 4.2]{Deb01}, hence $K_X$ is not pseudoeffective. Set
	$$ \delta := \inf \big\{ t \in \R_{\geq 0} \mid K_X+t\Delta+L \text{ is pseudoeffective} \big\} \in [0,1] . $$
	If $\delta>0$, then $K_X+\delta\Delta+L$ is num-effective by Theorem \ref{thm:GNV_psef_threshold_lc}, hence $K_X+\Delta+L$ is num-effective.
	
	From now on we assume that $\delta=0$. Consider
	$$ \lambda := \inf \big\{ t \in \R_{\geq 0} \mid K_X+tL \text{ is pseudoeffective} \big\} \in [0,1] . $$
	Then $\lambda>0$ since $K_X$ is not pseudoeffective, and hence
	\begin{equation}\label{eq:67a}
	K_X+\lambda L\quad\text{is num-effective}	
	\end{equation}	 
	by Theorem \ref{thm:GNV_psef_threshold_lc}. This immediately completes the proof of (a). 
	
	From now on, we concentrate on (d), and therefore we may assume that $\nu(X,K_X+\Delta+L)=1$.

	\medskip
	
	\emph{Step 5.}
	If $\lambda=1$, then we are done. Therefore, we may assume that $\lambda\in(0,1)$. Note that 
	$$\nu(X,K_X+L)\leq\nu(X,K_X+\Delta+L)=1.$$
	If $\nu(X,K_X+L)=0$, then $K_X+L$ is num-effective by \cite[Proposition V.2.7(8)]{Nak04}, hence $K_X+\Delta+L$ is num-effective. Thus, we may assume that 
	\begin{equation}\label{eq:67b}
		\nu(X,K_X+L)=1.	
	\end{equation}

	Fix a positive integer $m>2n$. Since $K_X+\lambda L$ is pseudoeffective, we have
	\begin{equation}\label{eq:67c}
	\nu(X,K_X+mL)=1
	\end{equation}
	by \eqref{eq:67b} and by Lemma \ref{lem:numdim}, since both divisors $K_X+L$ and $K_X+mL$ are positive linear combinations of the pseudoeffective divisors $K_X+\lambda L$ and $L$. By Proposition \ref{prop:MMP} we may run an $L$-trivial $K_X$-MMP $\varphi\colon X \dashrightarrow Y$ such that $K_Y+mL_Y$ is nef, where $L_Y:=\varphi_*L$. By \eqref{eq:67a} there exists an effective divisor $G$ on $ Y $ such that 
	\begin{equation}\label{eq:67d}
	K_Y+\lambda L_Y\equiv G,	
	\end{equation}
	and hence
	$$K_Y+mL_Y\equiv G+(m-\lambda)L_Y.$$
	
	If $G\neq0$, then by \cite[Theorem 6.1]{LP18a} we infer that $K_Y+mL_Y$ is num-effective, since $\nu(Y,K_Y+mL_Y)=1$ by \eqref{eq:67c}. Therefore, the divisor
	$$\textstyle K_Y+L_Y=\frac{m-1}{m-\lambda}(K_Y+\lambda L_Y)+\frac{1-\lambda}{m-\lambda}(K_Y+mL_Y)$$
	is num-effective, hence $K_X+L$ is num-effective by Lemma \ref{lem:MMPnumeff}. This shows that $K_X+\Delta+L$ is num-effective.
	
	Finally, assume that $G=0$. Then ${-}K_Y$ is nef by \eqref{eq:67d}, hence it is num-effective by Theorem \ref{thm:mainthm_dim=3}. In particular, the divisor $K_Y+L_Y\equiv\frac{1-\lambda}{\lambda}(-K_Y)$ is num-effective. Thus, $K_X+L$ is num-effective by Lemma \ref{lem:MMPnumeff} and consequently $K_X+\Delta+L$ is num-effective. This finishes the proof of (d) and of the theorem.
\end{proof}

\section{Proof of Theorem \ref{thm:mainthm_dim=4}}\label{sec:dim4nu1}

In this section we prove Theorem \ref{thm:mainthm_dim=4}. Similarly as in Section \ref{sec:TheoremA}, the main part of the argument was already completed in the previous sections and the following result deals with the remaining cases. Our proof uses similar ideas to those of \cite[Theorems 4.1 and 4.2]{LP17}.

\begin{prop} \label{pro:nu3}  
	Let $X$ be a projective variety with canonical singularities of dimension $4$ such that ${-}K_X$ is nef and $\chi(X,\OO_X)>0$. If $L$ is a nef Cartier divisor on $X$ with $\nu(X,L) \geq 3$, then
	$$\kappa(X,L)\geq0\quad\text{or}\quad \kappa(X,K_X+\ell L)\geq0\text{ for some integer }\ell\geq2.$$
\end{prop} 

\begin{proof}
By Theorem \ref{thm:dltblowup}(a) and Lemma \ref{lem:descenteff} we may assume that $X$ is $\Q$-factorial and terminal, and thus that its singular locus has dimension at most $1$ by \cite[Corollary 5.18]{KM98}. Let $\pi\colon Y \to X$ be a resolution which is an isomorphism over the smooth locus of $X$. Since $ \nu(Y, \pi^* L) = 3 $, we have $(\pi^*L)^4=0$, while by representing $K_X$ and $L$ as differences of ample divisors which are general in their $\Q$-linear systems, we also have
\begin{equation}\label{eq:dot}
	(\pi^*L)^3 \cdot K_Y = (\pi^*L)^3 \cdot \pi^*K_X \quad\text{and}\quad (\pi^*L)^2 \cdot K_Y^2 = (\pi^*L)^2 \cdot (\pi^*K_X)^2.
\end{equation}

Recall that the Todd class of $Y$ is
$$\operatorname{td}(Y)=1-\frac12K_Y+\frac{1}{12}\big(K_Y^2+c_2(Y)\big)-\frac{1}{24}K_Y\cdot c_2(Y)+\chi(Y,\OO_Y),$$
and note that $\chi(Y,\OO_Y)=\chi(X,\OO_X)\geq0$. Then the Hirze\-bruch-Riemann-Roch theorem, together with \eqref{eq:dot} and the fact that $X$ has rational singularities, gives for any integer $m$:
\begin{align}
\chi\big(X,&\OO_X(mL)\big)= \chi\big(Y,\OO_Y(m\pi^*L)\big)\label{eq:RR2}\\
=& -\frac{1}{12} m^3 (\pi^*L)^3\cdot \pi^*K_X+\frac{1}{24} m^2 (\pi^*L)^2 \cdot \big((\pi^*K_X)^2+c_2(Y)\big)\notag \\
&-\frac{1}{24} m (\pi^*L) \cdot K_Y \cdot c_2(Y) + \chi(X,\OO_X).\notag
\end{align}
We have
\begin{equation}\label{eq:1}
{-}(\pi^*L)^3\cdot \pi^*K_X\geq0\quad\text{and}\quad (\pi^*L)^2 \cdot (\pi^*K_X)^2\geq0
\end{equation}
since $L$ and ${-}K_X$ are nef, and 
\begin{equation}\label{eq:2}
(\pi^*L)^2 \cdot c_2(Y)\geq0
\end{equation}
by Theorem \ref{thm:c2psef}.

\medskip

Assume first that any of the inequalities in \eqref{eq:1} and \eqref{eq:2} is strict. Then
\begin{equation}\label{eq:0}
\chi\big(X,\OO_X(mL)\big)>0\quad \text{for }m\gg0.
\end{equation}
Since 
$$\nu(X,{-}K_X+mL)\geq\nu(X,mL)\geq3,$$
by Lemma \ref{lem:KVvanishing} we obtain
\begin{equation}\label{eq:vanish}
H^i\big(X,\OO_X(mL)\big) = H^i\big(X,\OO_X(K_X+({-}K_X+mL))\big)= 0\quad\text{ for }i\geq2.
\end{equation}
From \eqref{eq:0} and \eqref{eq:vanish} we deduce that
$$ h^0\big(X,\OO_X(mL)\big)>0\quad\text{for }m\gg0. $$ 

Therefore, we may assume that all inequalities in \eqref{eq:1} and \eqref{eq:2} are equalities. Then \eqref{eq:RR2} gives
$$ \chi\big(X,\OO_X(mL)\big)= \chi\big(Y,\OO_Y(m\pi^*L)\big)= -\frac{1}{24} m (\pi^*L) \cdot K_Y \cdot c_2(Y) + \chi(X,\OO_X). $$
If ${-}(\pi^*L) \cdot K_Y \cdot c_2(Y)\geq0$, then $\chi\big(X,\OO_X(mL)\big)\geq 1$ for any non-negative integer $ m $, and we conclude that $\kappa(X,L)\geq0$ as in the previous paragraph. 

We may thus assume that ${-}(\pi^*L) \cdot K_Y \cdot c_2(Y)<0$, so that
$$ \chi\big(Y,\OO_Y({-}m\pi^*L)\big)>0\quad \text{for }m\gg0, $$
hence by Serre duality,
\begin{equation}\label{eq:0a}
\chi\big(Y,\OO_Y(K_Y+m\pi^*L)\big)>0\quad \text{for }m\gg0.
\end{equation}
Since by Lemma \ref{lem:KVvanishing} we have
$$ H^i\big(X,\OO_Y(K_Y+m\pi^*L)\big) = 0\quad\text{ for }i\geq2, $$
from \eqref{eq:0a} we obtain
$$ h^0\big(Y,\OO_Y(K_Y + m\pi^*L)\big) >0 \quad \text{for }m\gg0. $$
Therefore, there exists an integer $\ell\geq2$ and an effective divisor $D$ such that $K_Y + \ell\pi^*L\sim D$. Pushing this relation to $X$ we obtain $K_X+\ell L\sim \pi_*D$, hence $\kappa(X,K_X+\ell L)\geq0$, as desired.
\end{proof} 

We are now ready to finish the proof of Theorem \ref{thm:mainthm_dim=4}.

\begin{proof}[Proof of Theorem \ref{thm:mainthm_dim=4}]
If $\nu(X,{-}K_X)=0$, then $K_X$ is torsion by \cite[Theorem 8.2]{Kaw85b}. Therefore, we may assume that $\nu(X,{-}K_X)>0$, so that the variety $ X $ is uniruled by Lemma \ref{lem:nefanticanimpliesuniruled}. If $\nu(X,{-}K_X)=4$, then ${-}K_X$ is big. If $\nu(X,{-}K_X)=1$, then ${-}K_X$ is num-effective by Theorem \ref{thm:mainthm_numdim=1partB}.

Now assume that $\nu(X,{-}K_X)=3$. If $X$ is rationally connected, then by Proposition \ref{pro:nu3} and by \eqref{eq:557a} we have $\kappa(X,{-}K_X)\geq0$ or there exists an integer $\ell\geq2$ such that
$$\kappa\big(X,(\ell-1)({-}K_X)\big)=\kappa\big(X,K_X+\ell({-}K_X)\big)\geq0.$$
If $X$ is not rationally connected, then we conclude by Theorem \ref{thm:nu2_pair_uniruled}.

Finally, if $\nu(X,{-}K_X)=2$ and $X$ is not rationally connected, 
then ${-}K_X$ is num-effective by the same proof as that of Theorem \ref{thm:mainthm_numdim=1partB}, where we only replace Theorem \ref{thm:mainthm_numdim=1partA} by Theorem \ref{thm:mainthm_dim=3}.
\end{proof}		
	
	\bibliographystyle{amsalpha}
	\bibliography{biblio}
	
\end{document}